\documentclass{article}

\usepackage{microtype}
\usepackage{graphicx}
\usepackage{subcaption}
\usepackage{booktabs} % for professional tables

\usepackage{hyperref}

\usepackage[accepted]{icml2026}

\usepackage{amsmath}
\usepackage{amssymb}
\usepackage{mathtools}
\usepackage{amsthm}

\usepackage[capitalize,noabbrev]{cleveref}

\newcommand{\bbE}{\mathbb{E}}

\newcommand{\bbP}{\mathbb{P}}

\newcommand{\Real}{{\mathbb R}}

\def\[#1\]{\begin{align}#1\end{align}} 

\theoremstyle{plain}
\newtheorem{theorem}{Theorem}[section]
\newtheorem{proposition}[theorem]{Proposition}
\newtheorem{lemma}[theorem]{Lemma}

\theoremstyle{definition}
\newtheorem{definition}[theorem]{Definition}

\theoremstyle{remark}

\usepackage[textsize=tiny]{todonotes}

\icmltitlerunning{Extremely Sparse Edge-Exchangeable Networks}

\begin{document}

\twocolumn[
  \icmltitle{A Generative Model for Extremely Sparse \\
  Edge-Exchangeable Networks}



  \icmlsetsymbol{equal}{*}

  \begin{icmlauthorlist}
    \icmlauthor{Valentin Kilian}{yyy}
  \end{icmlauthorlist}

  \icmlaffiliation{yyy}{Department of Statistics, University of Oxford}

  \icmlcorrespondingauthor{Valentin Kilian}{kilian@stats.ox.ac.uk}

  \icmlkeywords{Machine Learning, ICML}

  \vskip 0.3in
]



\printAffiliationsAndNotice{}  % no special notice (required even if empty)

\begin{abstract}
We propose a graph generative model for sequences of extremely sparse, edge-exchangeable networks. Models for sparse graphs often face a trade-off between desirable properties like exchangeability and the ability to capture the sparsity observed in real-world networks. While models based on vertex or edge exchangeability have successfully generated sparse graphs, achieving the "extremely sparse" regime, where the number of edges scales near-linearly with the number of nodes, has remained a challenge. Recently, a novel Completely Random Measure (CRM) was introduced, demonstrating that this rate could be achieved within the vertex-exchangeable framework of Caron and Fox. This paper extends this work by demonstrating that this new CRM can be integrated into the alternative edge-exchangeable framework to achieve extreme sparsity.
\end{abstract}

% !TEX root = article.tex

\section{Introduction}

The proliferation of large-scale network data across diverse scientific domains, from microbiology and ecology to economics and the social sciences, has revealed a set of unifying structural properties. Among these, \emph{sparsity} stands out as one of the most fundamental. For a simple graph with $N$ nodes, the number of possible edges scales quadratically, specifically as $\binom{N}{2} = \Theta(N^2)$. However, empirical networks rarely approach this theoretical maximum. We formalize this observation by defining a graph sequence as \emph{dense} if its edge count scales quadratically with its node count, and \emph{sparse} if the scaling is sub-quadratic. Real-world networks are overwhelmingly sparse. While they also display other important features like power-law degree distributions, the small-world phenomenon, and community structures, our focus here is squarely on the challenge of modeling sparsity.

Constructing a statistically coherent model for sparse networks is notoriously difficult, as sparsity is often incompatible with desirable modeling assumptions. For instance, nodes exchangeability—the property that a model's distribution is invariant to the relabelling of nodes—is a desirable statistical property for a generative model. However, the Aldous-Hoover theorem \cite{Aldous1981, Hoover1979} presents a major obstacle, demonstrating that any graph model satisfying this property must be either dense or empty (see also \citealp{Orbanz2015}).

To circumvent this limitation, several distinct modeling paradigms have been developed. \emph{Preferential attachment models} \cite{Barabasi1999a, Bollobas2001}, for example, abandon exchangeability entirely to generate sparse graphs.

Other lines of research have focused on weakening the exchangeability assumption. One prominent approach, initiated by Caron and Fox \cite{Caron2017}, substitutes nodes exchangeability with the more flexible notion of \emph{Kallenberg exchangeability} \cite{Kallenberg2005,Veitch2015, Borgs2018}. This framework has spurred a rich body of work modeling various network phenomena, such as overlapping communities \cite{Todeschini2020, Miscouridou2026}, clustering \cite{Caron2023}, dynamic evolution \cite{Naik2022}, and core-periphery structures \cite{Naik2021}. An alternative strategy involves replacing nodes exchangeability with \emph{edge exchangeability} \cite{Crane2015,Crane2016,Broderick2016,Cai2016b}. Models in this class also achieve sparsity. Both of these approaches employ Bayesian nonparametric tools like \emph{Completely Random Measures (CRMs)}. These models have been shown to generate graph sequences where the number of edges, $E$, scales as $N^{1+\epsilon}$ for $0 < \epsilon < 1$.

A recent breakthrough by \citet{Kilian2025} has redefined the achievable level of sparsity within the Caron-Fox framework. They introduced a novel CRM that produces graph sequences where $E = \Theta(N\ell(N))$, with $\ell$ being a slowly varying function. This behavior is termed \emph{extremely sparse}. This paper builds directly on that insight. Our central contribution is to adapt the methodology from \citet{Kilian2025} to achieve this same extreme sparsity within the edge-exchangeable model class.

The remainder of this paper is organized as follows. \cref{sec:edgeexch} provides a review of the edge-exchangeable network model. \cref{sec:sparsity} and \cref{sec:betaprocess} contain our primary results: we first generalize the sparsity theorem from \citet{Cai2016b} to encompass the extremely sparse regime, and then we introduce the Beta process with rapid variation, demonstrating its utility in constructing an extremely sparse, edge-exchangeable model. Finally, \cref{sec:simu} presents simulation studies validating our theoretical findings. All deferred proofs and definitions of asymptotic notation are available in the appendix.

%\\\\\\\\\\\\\\\\\\\\\\\\\\\\\\\\ Edge-exchangeable graph sequences
\section{Edge-exchangeable graph sequences}
\label{sec:edgeexch}

The notion of an \emph{edge-exchangeable graph} appeared around 2015, in the work of Crane and Dempsey \cite{Crane2015,Crane2016} and of Broderick, Cai, and Campbell \cite{Broderick2016,Cai2016b}, with some results presented in workshops prior to publication. In what follows, we are inspired by the presentation in \citet{Cai2016b}. This section provides an introduction to the topic; we refer the reader to the aforementioned references for more details.

\subsection{Permutation Invariance to Edge Arrival Order}

Let $(\mathcal{G}_t)_t$ be a sequence of (multi)graphs, where each graph $\mathcal{G}_t = (\mathcal{V}_t, \mathcal{E}_t)$ consists of a finite set of vertices $\mathcal{V}_t$ and a finite multiset of timestamped edges $\mathcal{E}_t$. We assume the sequence is \emph{projective} (or growing) i.e., $\mathcal{V}_{t} \subseteq \mathcal{V}_{t+1}$ and $\mathcal{E}_{t} \subseteq \mathcal{E}_{t+1}$.

Each (timestamped) edge $(e,s) \in \mathcal{E}_t$ is a tuple whose first entry is a set $e$ of two vertices in $\mathcal{V}_t$ and whose second entry is a timestamp $s$ corresponding to the step at which this edge first appeared. As in the Caron-Fox model, we are only interested in vertices involved in at least one edge. Consequently, we can define $\mathcal{V}_t$ as the set of all vertices that appear in the edges of $\mathcal{E}_t$, in which case the graph $\mathcal{G}_t$ is completely defined by its edge set $\mathcal{E}_t$.

\begin{definition}[\citealp{Cai2016b} Definition 2.4]
	Consider a random graph sequence $(\mathcal{G}_t)_t$ defined as above with $\mathcal{E}_t=\{ (e_1,s_1), \dots, (e_m,s_m)\}$. The sequence $(\mathcal{G}_t)_t$ is (infinitely) \emph{edge-exchangeable} if for every $t \in \mathbb{N}$ and for every permutation $\pi$ of the steps $\{1,\dots,t\}$,
	$\mathcal{G}_t \overset{d}{=} \tilde{\mathcal{G}}_t$, where $\tilde{\mathcal{G}}_t$ has the edge set $\tilde{\mathcal{E}}_t=\{ (e_1,\pi(s_1)), \dots, (e_m,\pi(s_m))\}$.
\end{definition}

In other word, the distribution of the generative model is invariant under any finite permutation of the order in which the edges arrived. This model (and its inherent symmetry) can be connected to partition, feature allocation and trait allocation; these connections are explained in \citet{Broderick2016} and \citet{Campbell2018a}.
	
\subsection{A Bayesian Nonparametric Model}

We now present a graph generative model that exhibits edge-exchangeability, which we will later prove can generate extremely sparse graphs. This model, referred to as the \emph{graph frequency model} in \citet{Cai2016b}, is similar to the Caron-Fox model from \citet{Caron2017} but with several key differences. These differences explain why one model is edge-exchangeable while the other is Kallenberg-exchangeable. A comparison of these two models is provided in \citet{Cai2016b}, where the authors also prove that the two notions are distinct.

In the graph frequency model, we consider a countably infinite set of latent vertices, indexed by the positive integers $\mathbb{N}$. Associated with these vertices is an infinite collection of edge labels $\{\theta_{\{i,j\}}\}$ with values in $[0,1]$ and edge probabilities $\{w_{\{i,j\}}\}$ with values in $[0,1]$. For any given (potentially random) choice of these labels and probabilities, we define the measure $G$ on $[0,1]$ as:
$$G=\sum_{\{i,j\} : i,j\in\mathbb{N}} w_{\{i,j\}}\delta_{\theta_{\{i,j\}}}.$$
If either the labels $\{\theta_{\{i,j\}}\}$ or the probabilities $\{w_{\{i,j\}}\}$ (or both) are random, then $G$ is a \emph{discrete random measure} on $[0,1]$. Given $G$, the graph sequence is constructed recursively. We initialize with $\mathcal{E}_0=\emptyset$. Then, at each step $t$, we form a new edge set, $F^t_{\text{new}}$, by sampling the multiplicity $m^t_{\{i,j\}}$ for every possible edge $\{i,j\}$ according to:
$$m^t_{\{i,j\}} \sim \text{Bernoulli}(w_{\{i,j\}})$$
We then add $m^t_{\{i,j\}}$ copies of the edge $\{i,j\}$ with timestamp $t$ to $F^t_{\text{new}}$. Finally, the edge set is updated as $\mathcal{E}_{t+1}=\mathcal{E}_t\cup F^t_{\text{new}}$. This process generates multigraphs, potentially adding multiple edges at each time step.

\begin{proposition}[\citealp{Cai2016b}]
The sequence of multigraphs generated via the preceding method is edge-exchangeable.
\end{proposition}

\begin{proof}
Conditional on $G$, the formation of an edge at any step $t$ is an independent event with an identical probability distribution. The exchangeability of the time stamps follows directly.
\end{proof}

This model can be implemented with a random measure $G$ constructed from a Poisson point process. Let $\mathcal{W}$ be a Poisson process on $[0, 1]$ with a non-atomic, $\sigma$-finite rate measure $\nu$ that satisfies $\nu([0, 1]) = \infty$ and $\int_0^1 w \,\nu(\mathrm{d}w) < \infty$\footnote{These two conditions on $\nu$ ensure that $\mathcal{W}$ is a countably infinite collection of weights in $[0,1]$ and that their sum $\sum_{w\in\mathcal{W}} w < \infty$ a.s.}. We can then use $\mathcal{W}$ to define the set of edge probabilities as $w_{\{i,j\}} = w_i w_j$ for $i\neq j$, and $w_{\{i,i\}}=0$ (to prevent self-loops), where each $w_i \in \mathcal{W}$. The edge labels $\theta_{\{i,j\}}$ can be sampled independently and uniformly from $[0,1]$. With this setup, $G$ becomes a homogeneous CRM on $[0,1]$ with no deterministic or fixed atomic components \cite{Kingman1967, Kingman1993, Lijoi2010}. The choice of the CRM $G$, and therefore the choice of the measure $\nu$, strongly influences the properties of the resulting edge-exchangeable graph sequence.

While this model generates multigraphs, it can be readily adapted to produce simple graphs $\overline{\mathcal{G}}_t=(\overline{\mathcal{V}}_t,\overline{\mathcal{E}}_t)$. This is done by setting $\overline{\mathcal{V}}_t=\mathcal{V}_t$ and defining $\overline{\mathcal{E}}_t$ as the set of unique edges in $\mathcal{E}_t$ (i.e., those with a multiplicity of at least one). Although the resulting simple graphs do not inherit the edge-exchangeability property, they retain many characteristics of the original multigraphs, most notably their sparsity. In what follows, we consider only multigraphs. 

\begin{figure}[H]
\centering
        \includegraphics[width=0.45\textwidth]{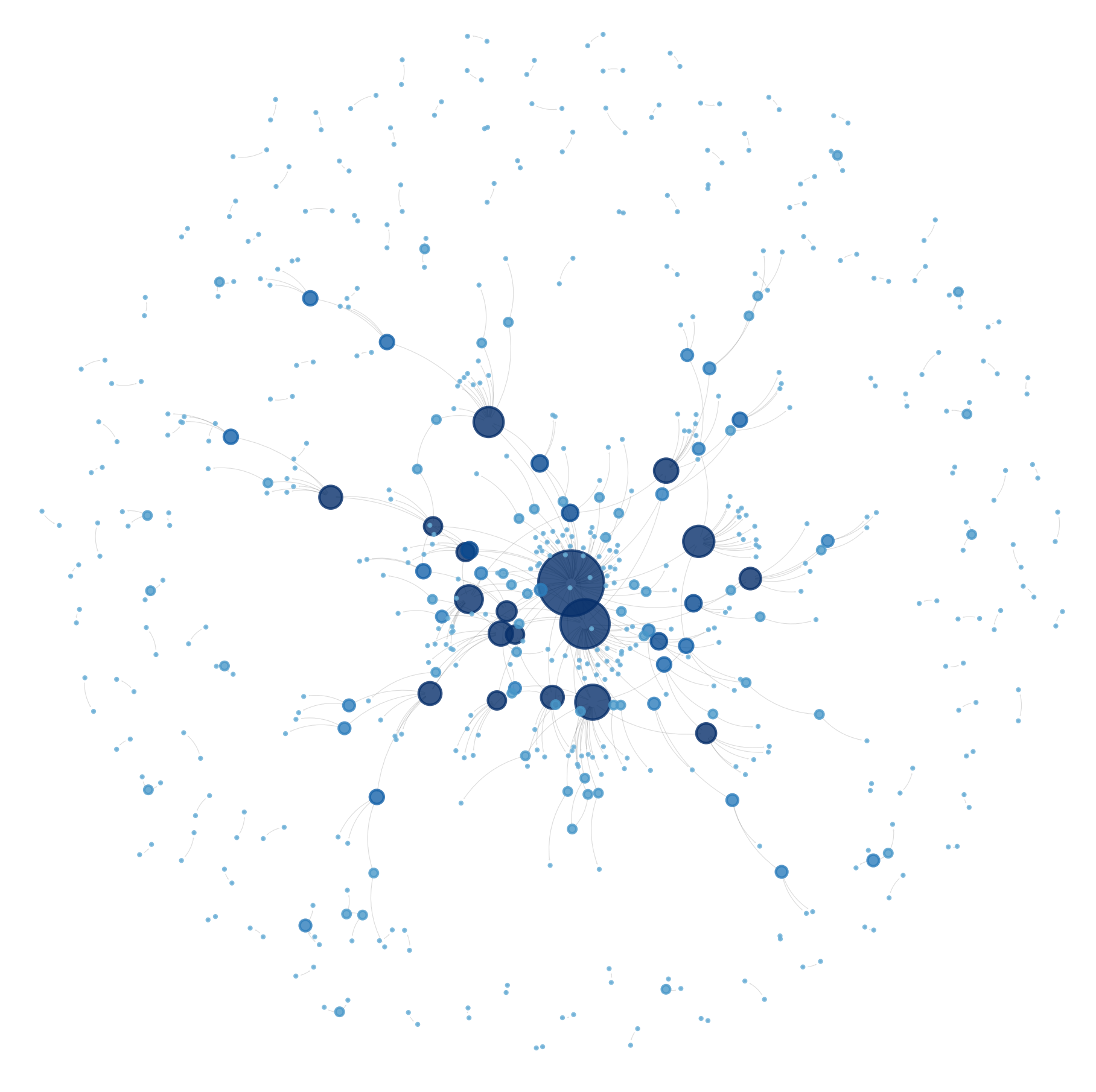}
   \caption{Edge-exchangeable graph generated from a RapidBeta process with $\alpha=1, \tau=0.95, \xi=1.2$, and $\eta=15$.}
\label{fig:graph}
\end{figure}

%///////////Extreme Sparsity
\section{Extreme Sparsity}
\label{sec:sparsity}

We denote the number of vertices and edges in the multigraph, respectively, as:
\begin{align*}
N_t &= |\mathcal{V}_t| = \sum_{i} \mathbf{1}_{\sum_{j\neq i} M^t_{\{i,j\}} > 0},\\
N^{(e)}_t &= |\mathcal{E}_t| = \frac{1}{2}\sum_{i\neq j} M^t_{\{i, j\}}, 
\end{align*}
where $M^t_{\{i, j\}} = \sum_{p=1}^t m^p_{\{i,j\}}$ is the multiplicity of the edge $\{i,j\}$ at time $t$. The following theorem states that an edge-exchangeable model constructed with a measure that has a rapidly varying tail is extremely sparse.

\begin{theorem}
\label{thm:sparsity}
Assume that, as \(w\downarrow0\),
\[
\nu(w)\sim c\,w^{-2}\ell(w^{-1}),
\]
where \(c>0\), \(\ell\) is slowly varying, and
\[
\ell_1(x):=\int_x^\infty u^{-1}\ell(u)\,\mathrm du<\infty
\]
for all sufficiently large \(x\).  If
$
\ell_1(x\ell_1(x))\sim \ell_1(x),
$
then
\[
N_t^{(e)}
=
\Theta\!\left(\frac{N_t}{\ell_1(N_t)}\right)
\qquad\text{a.s.}
\]
In particular, if
$
\ell(x)=(\log x)^a,
$ $a<-1$,
then \[
N_t^{(e)}
=
\Theta\!\left(N_t(\log N_t)^{-a-1}\right)
\qquad\text{a.s.}
\]
\end{theorem}

This result extends Theorem 5.3 in \citet{Cai2016b} to the extremely sparse regime. The main difference is that we consider rapidly varying measures, whereas they consider regularly varying measures that are not rapidly varying. This crucial distinction allows us to achieve a better sparsity rate. The proof methods rely on conditioning on the point process $\mathcal{W}$, after which the approach is similar to that of \cite{Janson2018}.

There exist several families of slowly varying functions that satisfy all the required assumptions in the theorem. The most useful one is the logarithmic family ${(\log x)^a}$ for $a<-1$. However, these also include more log-based families, such as the iterated log family $p{(\log_k (x))}^{-p-1}\prod_{i=0}^{k-1}{(\log_i (x))}^{-1}$ for $p>0$, where $\log_k$ is the logarithm iterated $k$ times, and the logarithmic exponential family $a{(\log x)}^{a-1}\exp(-{(\log x)}^a)$ for $0<a<1/2$. Additionally, there are Lambert $W$-based families, such as the Lambert $W$ family $\frac{pW{(x)}^{-p}}{1+W(x)}$ for $p>0$, where $W$ is the principal branch of the Lambert$W$ function, or more generally the iterated Lambert$W$ family $pW_k{(x)}^{-p}\prod_{i=1}^k\frac{1}{1+W_i (x)}$, where $W_k$ is the iterated principal branch of the Lambert $W$ function.

%////////////Beta process with rapid variation
\section{Beta process with rapid variation}
\label{sec:betaprocess}

We define the \emph{Beta process with rapid variation (RapidBeta)} as the process whose rate measure $\nu$ on $[0,1]$ has the density:
\begin{equation}
\nu(w) = \frac{\eta}{\alpha-\tau} \int_{\tau}^{\alpha} \frac{s}{\Gamma(1-s)}w^{-1-s} \left(1-w\right)^{\xi-1} \,\mathrm{d}s,
\end{equation}
where the parameters satisfy $\eta, \xi > 0$ and $0 \leq \tau < \alpha \leq 1$. For the special case where $\eta = \xi = 1$, this measure recovers the mixture of stable densities from \citet{Kilian2025}, restricted to the range $[0,1]$. \cref{fig:rapidbeta_density} displays the RapidBeta density for $\alpha=1$ and illustrates the effects of the parameters $\eta$, $\tau$, and $\xi$. The effect of $\alpha$ can be understood by adapting Proposition 1 from \citet{Kilian2025}:

\begin{figure}[h]
    \centering
     \includegraphics[width=0.6\textwidth]{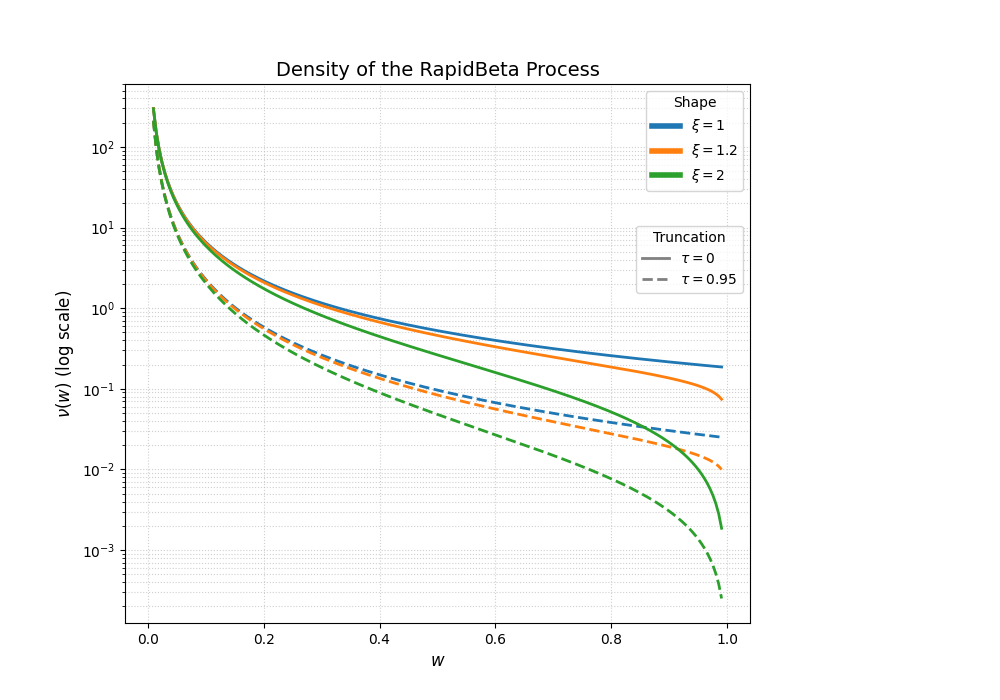}
    \caption{Density of the RapidBeta rate measure $\nu$ evaluated at $\alpha=1$ and $\eta=1$. Each parameter controls a distinct structural property of the process. The intensity $\eta$ acts as a global multiplier, scaling the overall mass. The shape parameter $\xi$ dictates the right-tail behavior. Higher values force the density to decay faster as $w \to 1$. Increasing the truncation parameter $\tau$ eliminates the lower-order terms in the asymptotic behavior of $\nu$ near zero.}
        \label{fig:rapidbeta_density}
\end{figure}

\begin{proposition}
\label{thm:AsymptoticsLevyLaplace}
If $\alpha=1$, the density of the rate measure satisfies: 
$$\nu(w) = \Theta\left(w^{-2}(\ln(1/w))^{-2}\right) \quad \text{as } w \to 0.$$
If $\alpha<1$, the density of the rate measure satisfies: 
$$\nu(w) = \Theta\left(w^{-1-\alpha}(\ln(1/w))^{-1}\right) \quad \text{as } w \to 0.$$
\end{proposition}

Additionally, one can verify that the two necessary conditions for the Poisson process construction hold: $\nu([0,1]) = \infty$ and $\int_0^1 w\nu(\mathrm{d}w) < \infty$ (see appendix for the proof). It follows that an edge-exchangeable graph sequence constructed using the RapidBeta measure with $\alpha=1$ satisfies the condition of Theorem~\ref{thm:sparsity}. Therefore, this construction yields an \emph{extremely sparse graph sequence}.

While the asymptotic sparsity regime of the RapidBeta construction is entirely determined by $\alpha$, the lower endpoint $\tau$ plays a non-negligible role at finite sample sizes. Indeed, the Lévy density is obtained by integrating contributions of the form $w^{-1-s}$ over $ s\in[\tau,\alpha]$, so that components with larger $s$ (i.e., closer to 1) dominate the small-$w$ behavior responsible for the asymptotic regime. However, when $\tau$ is substantially below 1, the mixture includes terms with smaller exponents, which decay more slowly and can materially influence the effective behaviour over the range of weights accessible in finite simulations. As a consequence, although all choices $\tau<1$ with $\alpha=1$ are asymptotically equivalent and lead to extreme sparsity, selecting $\tau$ close to 1 concentrates the mixture on exponents near the critical boundary and suppresses lower-order contributions, thereby accelerating the emergence of the extreme-sparse regime in practice.

%///////////////////////////Simulation////////////////////////:
\section{Simulation of the RapidBeta Process}
\label{sec:sampling}

To simulate an edge-exchangeable graph sequence as described in \cref{sec:edgeexch}, we must first generate the Poisson point process
 $$
 \mathcal{W} = \{w_i\}
 $$
 
When $ \mathcal{W}$ is a RapidBeta  process, the associated Lévy measure $\nu$ satisfies $\nu([0,1]) = \infty$, implying that $\mathcal{W}$ is almost surely infinite. Consequently, exact simulation is impossible, and we instead consider a truncated version.

\subsection{Truncated RapidBeta Process}

For a fixed $\varepsilon>0$, define the truncated Poisson process
\begin{equation}
 \mathcal{W} _\varepsilon=\{w_i \in\mathcal{W}: w_i > \varepsilon\},
\end{equation}
which contains only finitely many atoms a.s. Let us define the intensity measure $\lambda$ as
\begin{equation}
\lambda(\mathrm{d}s,\mathrm{d}x)
= \frac{\eta}{\alpha-\tau}\frac{s}{\Gamma(1-s)}x^{-1-s}\left(1-x\right)^{\xi-1}\mathrm{d}s\mathrm{d}x,
\end{equation}
Sampling from $\mathcal{W}$ is equivalent to sampling a 2D Poisson point process on $[\tau,\alpha]\times[0,1]$ with intensity $\lambda$, and discarding the first dimension. Sampling from $\mathcal{W}_\varepsilon$ is performed analogously but with the 2D Poisson point process taken on $[\tau,\alpha]\times[\varepsilon,1]$.

\subsection{Partitioned Thinning Scheme}

To construct an efficient exact sampler for $\mathcal{W}_\epsilon$, we partition $[\varepsilon,1]$ into two regions:
\[
A_\varepsilon = [\varepsilon, 1-\delta],
\qquad
B = (1-\delta, 1],
\]
where $\delta \in (0,1)$ is fixed. On $A_\varepsilon$, we use the bound
\begin{equation}
\lambda(\mathrm{d}s,\mathrm{d}x)
\le
\frac{\eta}{\alpha-\tau} B_A x^{-1-\alpha}
\,\mathrm{d}s\,\mathrm{d}x,
\end{equation}
where $B_A = \max\{1, \delta^{\xi-1}\}$. On $B$, we use the bound
\begin{equation}
\lambda(\mathrm{d}s,\mathrm{d}x)
\le
\frac{\eta}{\alpha-\tau} (1-\delta)^{-1-\alpha}
(1-x)^{\xi-1}
\,\mathrm{d}s\,\mathrm{d}x.
\end{equation}

These bounds define valid dominating measures on each region, allowing exact sampling via thinning as described in detail in \cref{alg:rapidbeta}.

\begin{algorithm}[ht]
\caption{Exact sampler for the truncated weight set $\mathcal{W}_\varepsilon = \{w_i > \varepsilon\}$}
\label{alg:rapidbeta}
\begin{algorithmic}[1]

\REQUIRE Parameters $\eta, \xi, \tau, \alpha$, truncation level $\varepsilon$, partition parameter $\delta$

\STATE $B_A \gets \max\{1, \delta^{\xi-1}\}$

\STATE Compute\\
$
\Lambda_A = \eta B_A \int_{\varepsilon}^{1-\delta} x^{-1-\alpha} \,\mathrm{d}x,
$\\
$
\Lambda_B = \eta (1-\delta)^{-1-\alpha} \int_{1-\delta}^{1} (1-x)^{\xi-1} \,\mathrm{d}x
$

\STATE Sample $N_A \sim \mathrm{Poisson}(\Lambda_A)$ and $N_B \sim \mathrm{Poisson}(\Lambda_B)$

\FOR{$k = 1,\dots,N_A$}
    \STATE Sample $S \sim \mathrm{Unif}[\tau,\alpha]$
    \STATE Sample $W$ with density proportional to $w^{-1-\alpha}$ on $(\varepsilon, 1-\delta]$:\\
    $U\sim\mathrm{Unif}[0,1]$\\
    $
    W =
    \left(
    \varepsilon^{-\alpha}
    - U \left(\varepsilon^{-\alpha} - (1-\delta)^{-\alpha}\right)
    \right)^{-1/\alpha}
    $
    \STATE Accept with probability
    $
    \frac{S}{\Gamma(1-S)}
    \frac{W^{\alpha-S}(1-W)^{\xi-1}}{B_A}
    $
    \IF{accepted}
        \STATE Add weight $W$ to $\mathcal{W}_\varepsilon$
    \ENDIF
\ENDFOR

\FOR{$k = 1,\dots,N_B$}
    \STATE Sample $S \sim \mathrm{Unif}[\tau,\alpha]$
    \STATE Sample $U\sim\mathrm{Unif}[0,1]$, set $W = 1 - \delta U^{1/\xi}$
    \STATE Accept with probability
    $
    \frac{S}{\Gamma(1-S)}
    \frac{W^{-1-S}}{(1-\delta)^{-1-\alpha}}
    $
    \IF{accepted}
        \STATE Add weight $W$ to $\mathcal{W}_\varepsilon$
    \ENDIF
\ENDFOR

\STATE \textbf{return} $\mathcal{W}_\varepsilon$

\end{algorithmic}
\end{algorithm}

\subsection{Correctness}

\begin{theorem}
\label{thm:correctness}
The set of weights $\mathcal{W}_\varepsilon$ produced by the thinning procedure described above is a realization of a Poisson point process on $[\varepsilon,1]$ with intensity $\nu(w)\,\mathrm{d}w$ i.e. the $\varepsilon$-truncated RapidBeta process.
\end{theorem}

\subsection{Truncation error induced by the $\varepsilon$-approximation}
\label{sec:truncation_error}

The $\varepsilon$-approximation replaces the full weight collection
$\mathcal{W}=\{w_i\}$ by the truncated set $\mathcal{W}_\varepsilon$. A natural measure of the approximation error is the discarded total mass
\[
R_\varepsilon := \sum_{w_i \in \mathcal{W}: w_i \le \varepsilon} w_i.
\]

\begin{proposition}[Discarded mass under truncation]
\label{prop:discarded_mass}
Assume $\alpha=1$, then, as $\varepsilon \rightarrow 0$,
\[
\mathbb{E}[R_\varepsilon]
=
\Theta\!\left(\frac{1}{\ln(1/\varepsilon)}\right).
\]
In particular, the total mass discarded by truncating the CRM at level
$\varepsilon$ vanishes at a logarithmic rate.
\end{proposition}

\subsection{Diagnostics}

To asses the quality of our algorithm in practice we perform two diagnostics. The first diagnostic compares the empirical tail counts $N(>x)=\sum_i\mathbf{1}_{W_i>x}$, averaged over repeated samples, to the theoretical tail measure $\bar{\nu}(x)=\int_x^1\nu(w)\mathrm{d}w$, leveraging the property that $\mathbb{E}[N(>x)] = \bar{\nu}(x)$ for a Poisson process. Agreement between empirical and theoretical curves indicates that the marginal tail behavior and intensity are correctly captured. The second diagnostic examines the transformed order statistics by plotting $\bar{\nu}(W_{(k)})$ against the rank k. For a correctly specified Poisson point process, these transformed values align with the identity line, reflecting the fact that the ordered points map to approximately unit-rate arrivals. Together, these diagnostics provide complementary validation: the former targets first-moment properties of the tail, while the latter probes the global distributional and structural consistency of the point process. In practice, our algorithm provides good performance; see \cref{fig:diagnosis}.

\begin{figure}[t]
    \centering
    \includegraphics[width=0.45\textwidth]{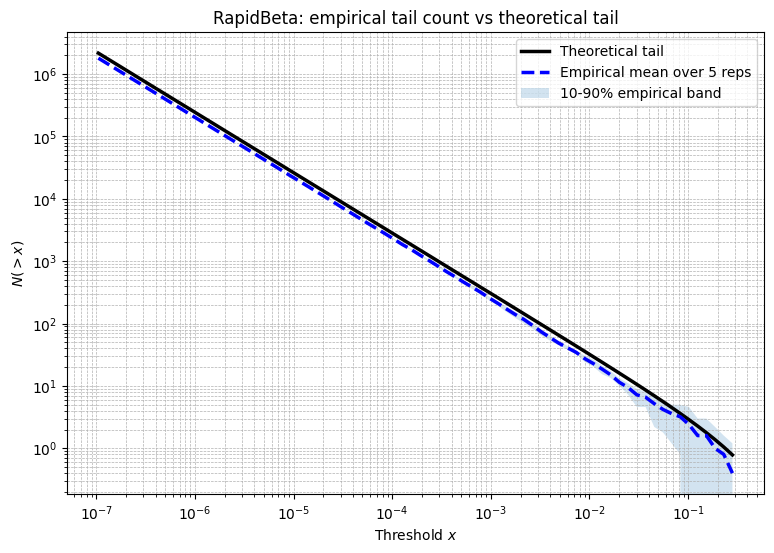}
        \includegraphics[width=0.45\textwidth]{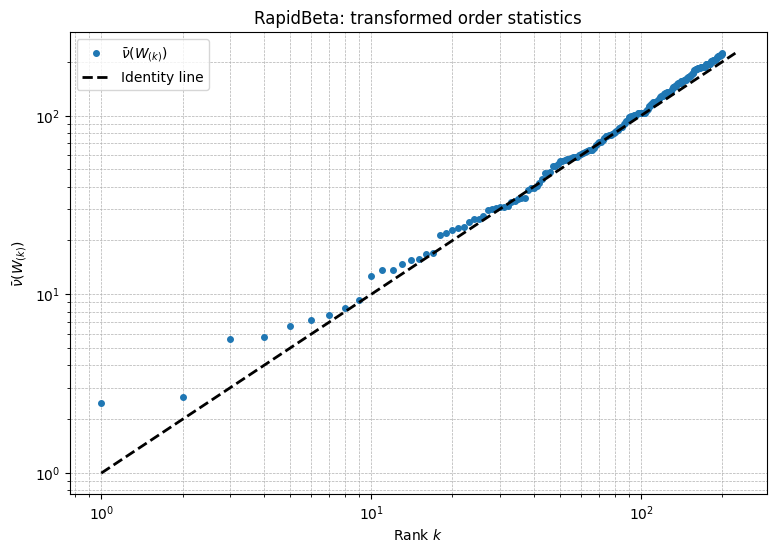}
       \caption{Diagnostics of the performance of our sampling algorithm, computed over 5 repetitions for a RapidBeta process with $\alpha=1, \tau=0.95, \xi=1.2$, and $\eta=15$ and $\varepsilon=10^{-7}$. We observe good performance.}
\label{fig:diagnosis}
\end{figure}

%//////////Simulations
\section{Simulations}
\label{sec:simu}

Once the weight generation procedure is available, we can generate the sequence of graphs following the graph frequency model (see \cref{alg:graph_simulation}).

\begin{algorithm}[t]
\caption{Graph Frequency Generative Model}
\label{alg:graph_simulation}
\begin{algorithmic}[1]
\STATE \textbf{Input:}  Weights $\mathbf{w} = (w_1, \dots, w_N)$; Time step $T$.
\STATE \textbf{Output:} An $N \times N$ edge multiplicity matrix $\mathbf{A}$.

\vspace{0.5em}
\STATE Let $N$ be the number of nodes, i.e., the length of $\mathbf{w}$.
\STATE Initialize $\mathbf{A}$ as an $N \times N$ matrix of zeros.

\FOR{$i = 1, \dots, N$}
    \FOR{$j = i + 1, \dots, N$}
        \STATE $p_{ij} \gets w_i w_j$ 
        \COMMENT{Connection probability}
        \STATE $A_{ij} \sim \text{Binomial}(T, p_{ij})$ 
        \COMMENT{Number of edges}
        \STATE $A_{ji} \gets A_{ij}$ 
        \COMMENT{Ensure symmetry for an undirected graph}
    \ENDFOR
\ENDFOR

\STATE \textbf{return} $\mathbf{A}$
\end{algorithmic}
\end{algorithm}

We compare our method to graph sequences generated by a three-parameter Beta process, as described in \cite{Cai2016b}. We sample the weights of the three-parameter Beta process using a Poisson thinning construction based on its Lévy measure. Recall that the Beta process is a discrete random measure whose atom weights are generated by a Poisson point process with Lévy density
\begin{equation}
 \nu(\mathrm{d}w) = \frac{\gamma\Gamma(1+\beta)}{\Gamma(1-\alpha)\Gamma(\beta+\alpha)} w^{-1-\alpha}(1-w)^{\beta+\alpha-1}\mathrm{d}w
\end{equation}
where $\alpha \in (0,1)$ and $\beta,\gamma>0$. We simulate the associated Poisson process by first drawing candidate jumps from a dominating stable Lévy density proportional to $w^{-1-\alpha}$, whose inverse tail admits a closed-form expression. Each proposed jump is then accepted with probability $(1-w)^{\beta+\alpha-1}$, yielding an exact thinning scheme whenever $\beta+\alpha-1 \ge 0$. This approach avoids the numerical instabilities of stick-breaking representations \cite{Broderick2012a} when $\alpha$ is close to one, while preserving the correct distribution of large weights.

The simulations are displayed in \cref{fig:comparaison}, as expected, the RapidBeta process produces sparser graph sequences than the three Beta process. We also provide a comparison to the Barabási-Albert model \cite{Barabasi1999a}, which produces even sparser sequences ($N^{(e)}_t = \Theta(N_t)$) but lacks exchangeability.

\begin{figure*}[h!]
    \centering
    \begin{subfigure}[b]{0.48\textwidth}
        \centering
        \includegraphics[width=\textwidth]{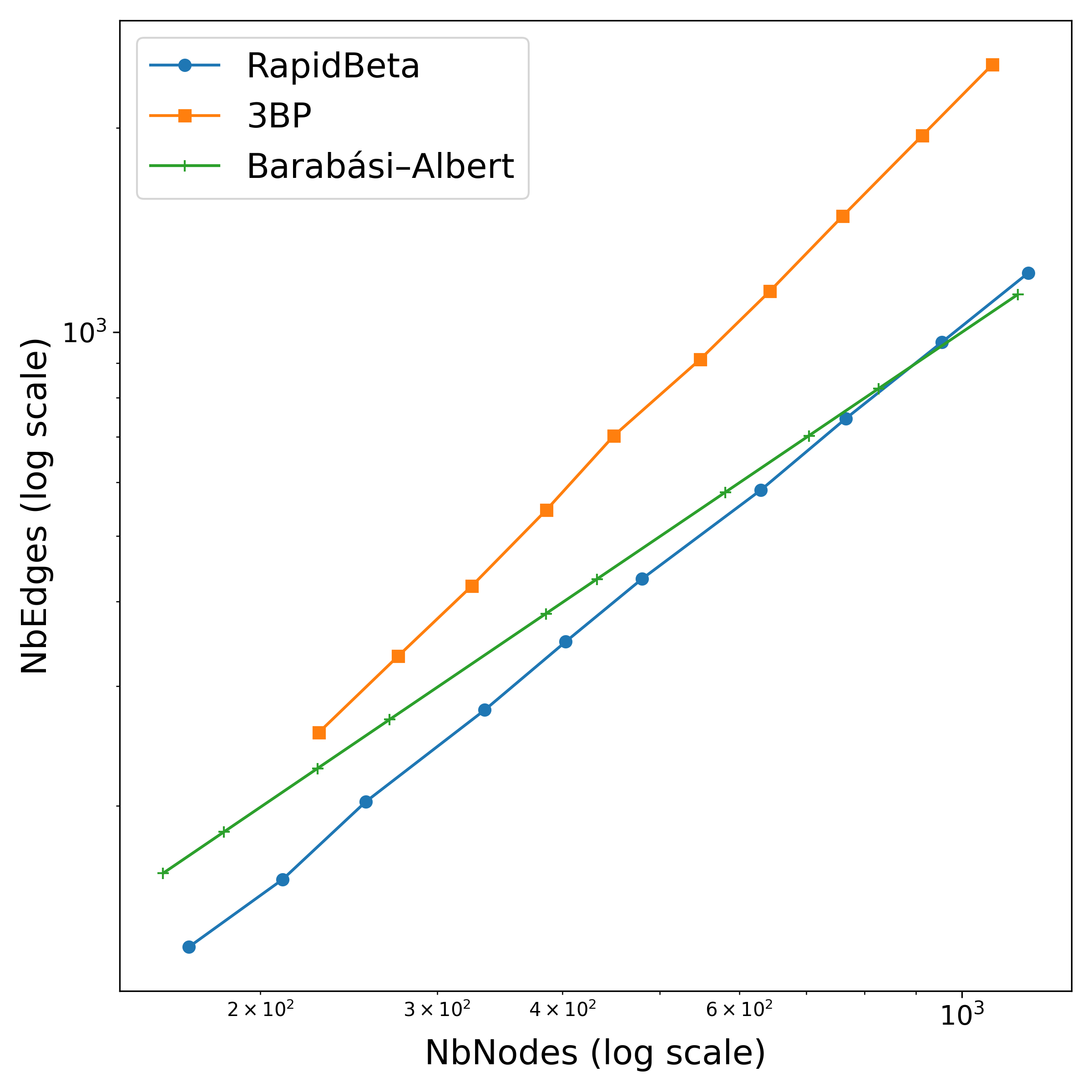}
    \end{subfigure}
    \hfill
    \begin{subfigure}[b]{0.48\textwidth}
        \centering
        \includegraphics[width=\textwidth]{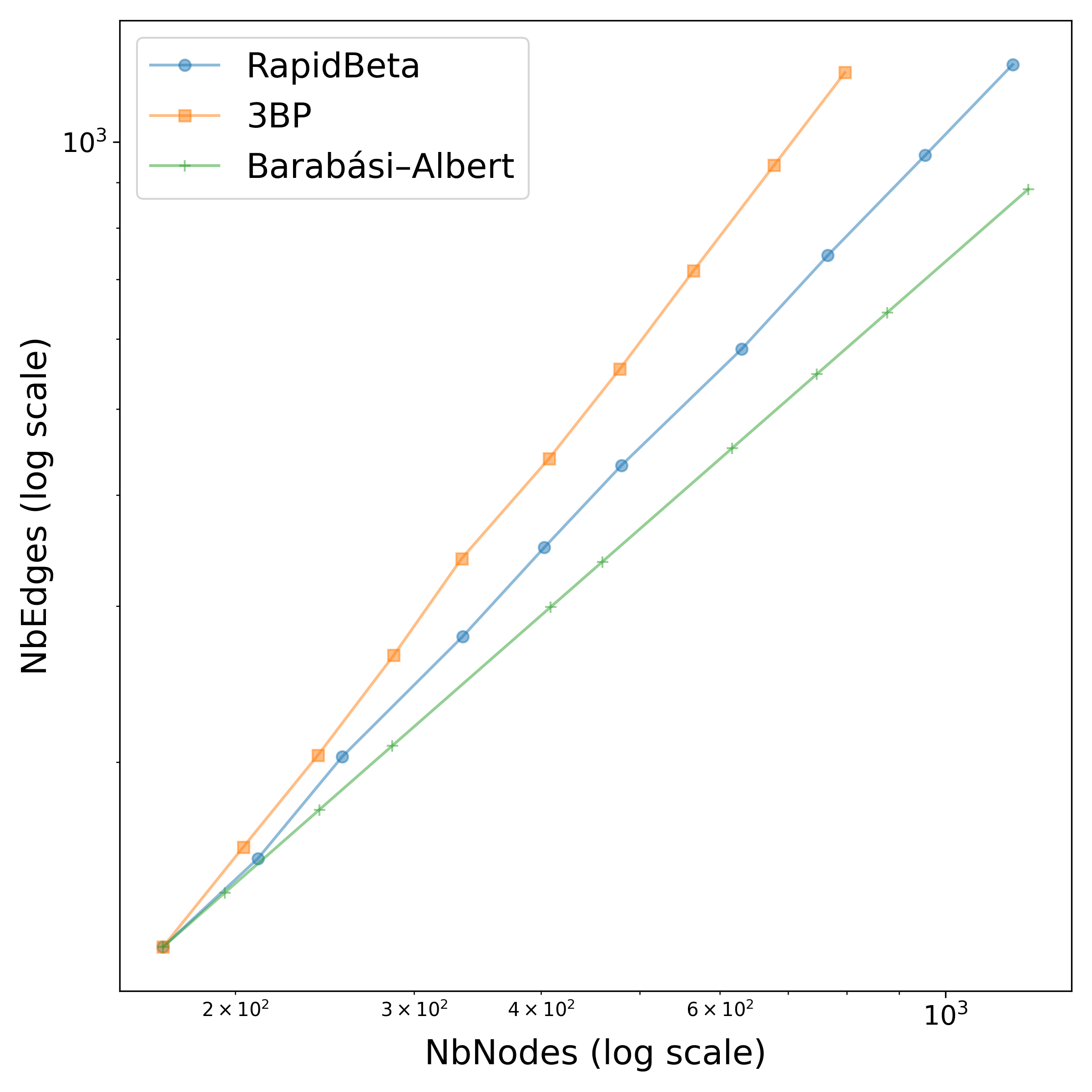}
    \end{subfigure}
   \caption{(Right: number of edges ($N_t^{(e)}$) versus number of nodes ($N_t$) for the \textbf{RapidBeta} model ({\Large$\bullet$}) was simulated with parameters $\alpha=1, \tau=0.95, \xi=1.2$, and $\eta=15$. This is compared against the \textbf{three-parameter Beta process} model ({\footnotesize$\blacksquare$}) with parameters $\gamma=4, \beta=4, \alpha=0.7$, and the \textbf{Barabási–Albert} model ($\blacktriangle$). Simulations were run for $t$ ranging from 30 to 300. Each point represents the mean over ten independent graph sequence samples. Left: same plot with all data shifted so the first points coincide which eases the comparison of the slopes in log-log scale.}
\label{fig:comparaison}
\end{figure*}

The code used to generate the edge-exchangeable graphs can be found at: \url{https://github.com/ValentinKil/ExtremelySparseEdgesExchangeable.git}.

\section{Conclusion}

In this work, we presented a graph generative model for edge-exchangeable networks. We first showed that the extremely sparse regime can be achieved through a sequence of edge-exchangeable graphs. We then proposed a practical generative model that enables the simulation of these sequences. Our experimental results demonstrate that, despite necessary computational approximations, the model effectively captures the extremely sparse regime.

While \citet{Li2021} provide rigorous quantitative bounds on the error for fixed-rank truncation of edge-exchangeable networks, equivalent theoretical guarantees for the threshold $\varepsilon$-truncation employed in our sampler are not yet available. Extending their error analysis to weight-based thresholds is a promising direction for future work.

Extreme sparsity for edge-exchangeable simple graphs has previously been studied by \citet{Janson2018}, who proposed two examples that reach exact linear dependence ($N^{(e)}_t = \Theta(N_t)$). However, the produced graphs exhibit some unwanted properties, such as bounded degrees or a structure composed mostly of stars. By relaxing the strict linearity requirement and accepting linear scaling \textit{up to a slowly varying function}, we are able to generate a much richer, more realistic internal graph structure with the RapidBeta process; indeed, experiments show the emergence of a giant connected component. A direct application of Theorem 14 in \citet{Eriksson2025} demonstrates that the RapidBeta graphs are almost surely not eventually forever connected, meaning there will always be some edges that do not join the giant component. To the best of our knowledge, there are no more precise results regarding the giant component of edge-exchangeable graphs, although \citet{Eriksson2025} suggests some interesting directions.

Other properties of the generative model, such as the power-law behaviour of the degree distribution, the clustering coefficient, and the small-world property, remain to be explored. Additionally, developing inference procedures for RapidBeta edge-exchangeable graphs is a compelling avenue for future research.

\section*{Acknowledgements}
I would like to thank François Caron for his insightful advice and thorough proofreading of this manuscript. I am funded by the Clarendon Scholarship.

\bibliographystyle{icml2026}
\bibliography{edgeexch}

@incollection{Lijoi2010,
  author =        {Lijoi, A. and Pr{\"u}nster, I.},
  booktitle =     {Bayesian Nonparametrics},
  editor =        {Hjort, N. L. and Holmes, C. and M{\"u}ller, P. and
                   Walker, S. G.},
  publisher =     {Cambridge University Press},
  title =         {Models beyond the {{Dirichlet}} Process},
  year =          {2010},
}

@article{Caron2017,
  author =        {Caron, F. and Fox, E.},
  journal =       {Journal of the Royal Statistical Society. Series B
                   (Statistical Methodology)},
  pages =         {1295--1366},
  title =         {Sparse Graphs Using Exchangeable Random Measures},
  volume =        {79},
  year =          {2017},
}

@article{Gnedin2007,
  author =        {Gnedin, A. and Hansen, B. and Pitman, J.},
  journal =       {Probab. Surv},
  number =        {146-171},
  pages =         {88},
  title =         {Notes on the Occupancy Problem with Infinitely Many
                   Boxes: General Asymptotics and Power Laws},
  volume =        {4},
  year =          {2007},
}

@article{Caron2023,
  author =        {Caron, Fran{\c c}ois and Panero, Francesca and
                   Rousseau, Judith},
  journal =       {Advances in Applied Probability},
  month =         dec,
  number =        {4},
  pages =         {1211--1253},
  title =         {On Sparsity, Power-Law, and Clustering Properties of
                   Graphex Processes},
  volume =        {55},
  year =          {2023},
  issn =          {0001-8678, 1475-6064},
}

@article{Kingman1967,
  author =        {Kingman, J.F.C.},
  journal =       {Pacific Journal of Mathematics},
  number =        {1},
  pages =         {59--78},
  publisher =     {Pacific Journal of Mathematics, A Non-profit
                   Corporation},
  title =         {Completely Random Measures},
  volume =        {21},
  year =          {1967},
}

@book{Kingman1993,
  author =        {Kingman, J.F.C.},
  publisher =     {Oxford University Press, USA},
  title =         {Poisson Processes},
  volume =        {3},
  year =          {1993},
}

@article{Veitch2015,
  author =        {Veitch, Victor and Roy, Daniel M.},
  month =         dec,
  journal =        {arXiv:1512.03099},
  publisher =     {arXiv},
  title =         {The {{Class}} of {{Random Graphs Arising}} from
                   {{Exchangeable Random Measures}}},
  year =          {2015},
}

@article{Borgs2018,
  title={Sparse exchangeable graphs and their limits via graphon processes},
  author={Borgs, Christian and Chayes, Jennifer T and Cohn, Henry and Holden, Nina},
  journal={JMLR},
  volume={18},
  number={210},
  pages={1--71},
  year={2018}
}

@article{Barabasi1999a,
  title={Emergence of scaling in random networks},
  author={Barab{\'a}si, Albert-L{\'a}szl{\'o} and Albert, R{\'e}ka},
  journal={Science},
  volume={286},
  number={5439},
  pages={509--512},
  year={1999},
  publisher={American Association for the Advancement of Science}
}

@article{Todeschini2020,
  author =        {Todeschini, Adrien and Miscouridou, Xenia and
                   Caron, Fran{\c c}ois},
  journal =       {Journal of the Royal Statistical Society: Series B
                   (Statistical Methodology)},
  number =        {2},
  pages =         {487--520},
  title =         {Exchangeable Random Measures for Sparse and Modular
                   Graphs with Overlapping Communities},
  volume =        {82},
  year =          {2020},
  doi =           {10.1111/rssb.12363},
  issn =          {1467-9868},
}

@article{Naik2021,
  author =        {Naik, Cian and Caron, Fran{\c c}ois and
                   Rousseau, Judith},
  journal =       {Electronic Journal of Statistics},
  month =         jan,
  number =        {1},
  pages =         {1814--1868},
  publisher =     {{Institute of Mathematical Statistics and Bernoulli
                   Society}},
  title =         {Sparse Networks with Core-Periphery Structure},
  volume =        {15},
  year =          {2021},
  issn =          {1935-7524, 1935-7524},
}

@article{Janson2018,
  author =        {Janson, Svante},
  journal =       {Journal of Statistical Physics},
  month =         nov,
  number =        {3},
  pages =         {448--484},
  title =         {On {{Edge Exchangeable Random Graphs}}},
  volume =        {173},
  year =          {2018},
  doi =           {10.1007/s10955-017-1832-9},
  issn =          {1572-9613},
}

@book{Bingham1987,
  author =        {Bingham, N. H and Goldie, C. M. and Teugels, J. L.},
  publisher =     {Cambridge university press},
  title =         {Regular Variation},
  volume =        {27},
  year =          {1987},
}

@book{Feller1971,
  author =        {Feller, W.},
  publisher =     {John Wiley \& Sons},
  title =         {An Introduction to Probability Theory and Its
                   Applications},
  volume =        {2},
  year =          {1971},
}

@article{Kilian2025,
  title = {Rapidly {{Varying Completely Random Measures}} for {{Modeling Extremely Sparse Networks}}},
  author = {Kilian, Valentin and Guedj, Benjamin and Caron, Fran{\c c}ois},
  year = {2025},
  month = may,
  journal = {arXiv:2505.13206},
  publisher = {arXiv},
}

@article{Broderick2016,
  title = {Edge-Exchangeable Graphs and Sparsity},
  author = {Broderick, Tamara and Cai, Diana},
  year = {2016},
  month = mar,
  journal = {arXiv:1603.06898},
  eprint = {1603.06898},
  primaryclass = {math},
  publisher = {arXiv},
  urldate = {2025-07-10},
  archiveprefix = {arXiv}
}

@inproceedings{Cai2016b,
  title = {Edge-Exchangeable Graphs and Sparsity},
  booktitle = {Advances in {{Neural Information Processing Systems}}},
  author = {Cai, Diana and Campbell, Trevor and Broderick, Tamara},
  year = {2016},
  volume = {29},
  urldate = {2025-07-10}
}

@article{Crane2015,
  title = {A Framework for Statistical Network Modeling},
  author = {Crane, Harry and Dempsey, Walter},
  year = {2015},
  month = sep,
  journal = {arXiv:1509.08185},
  eprint = {1509.08185},
  primaryclass = {math},
  publisher = {arXiv},
  urldate = {2025-07-10},
  archiveprefix = {arXiv}
}

@article{Crane2016,
  title = {Edge Exchangeable Models for Network Data},
  author = {Crane, Harry and Dempsey, Walter},
  year = {2016},
  month = oct,
  journal = {arXiv:1603.04571},
  eprint = {1603.04571},
  primaryclass = {math},
  publisher = {arXiv},
  urldate = {2025-07-10},
  archiveprefix = {arXiv}
}

@article{Aldous1981,
  title = {Representations for Partially Exchangeable Arrays of Random Variables},
  author = {Aldous, David J.},
  year = {1981},
  month = dec,
  journal = {Journal of Multivariate Analysis},
  volume = {11},
  number = {4},
  pages = {581--598},
  issn = {0047-259X},
  doi = {10.1016/0047-259X(81)90099-3},
  urldate = {2024-09-03}
}

@misc{Hoover1979,
  title = {Relations on {{Probability Spaces}} and {{Arrays}} of {{Random Variables}}},
  author = {Hoover, Douglas N.},
  year = {1979},
  number = {NJ08540},
  eprint = {NJ08540},
  publisher = {Institute for Advanced Study, Princeton},
  archiveprefix = {Institute for Advanced Study, Princeton}
}

@article{Orbanz2015,
  title = {Bayesian {{Models}} of {{Graphs}}, {{Arrays}} and {{Other Exchangeable Random Structures}}},
  author = {Orbanz, Peter and Roy, Daniel M.},
  year = {2015},
  month = feb,
  journal = {IEEE Transactions on Pattern Analysis and Machine Intelligence},
  volume = {37},
  number = {2},
  pages = {437--461},
  issn = {0162-8828, 2160-9292},
  doi = {10.1109/TPAMI.2014.2334607},
  urldate = {2024-02-06},
  langid = {english}
}

@misc{Naik2022,
  title = {Bayesian {{Nonparametrics}} for {{Sparse Dynamic Networks}}},
  author = {Naik, Cian and Caron, Francois and Rousseau, Judith and Teh, Yee Whye and Palla, Konstantina},
  year = {2022},
  month = apr,
  number = {arXiv:1607.01624},
  eprint = {1607.01624},
  primaryclass = {stat},
  publisher = {arXiv},
  doi = {10.48550/arXiv.1607.01624},
  urldate = {2025-07-10},
  archiveprefix = {arXiv}
}

@article{Broderick2012a,
  title = {Beta {{Processes}}, {{Stick-Breaking}} and {{Power Laws}}},
  author = {Broderick, Tamara and Jordan, Michael I. and Pitman, Jim},
  year = {2012},
  month = jun,
  journal = {Bayesian Analysis},
  volume = {7},
  number = {2},
  pages = {439--476},
  publisher = {International Society for Bayesian Analysis},
  issn = {1936-0975, 1931-6690},
  urldate = {2025-07-16}
}

@book{Bollobas2001,
  title = {Random {{Graphs}}},
  author = {Bollob{\'a}s, B{\'e}la},
  year = {2001},
  month = aug,
  publisher = {Cambridge University Press},
  googlebooks = {o9WecWgilzYC},
  isbn = {978-0-521-79722-1},
  langid = {english}
}

@book{Kallenberg2005,
  title = {Probabilistic {{Symmetries}} and {{Invariance Principles}}},
  author = {Kallenberg, Olav},
  year = {2005},
  month = dec,
  publisher = {Springer Science \& Business Media},
  isbn = {978-0-387-28861-1},
  langid = {english}
}

@article{Miscouridou2026,
  author = {Miscouridou, Xenia and Panero, Francesca and Laos, Antreas},
  month = apr,
  journal = {arXiv:2512.10717},
  publisher = {arXiv},
  title = {Dynamic Sparse Graphs with Overlapping Communities},
  year = 2026,

}

@article{Eriksson2025,
  title = {Edge {{Exchangeable Graphs}}: {{Connectedness}}, {{Gaussianity}} and {{Completeness}}},
  shorttitle = {Edge {{Exchangeable Graphs}}},
  author = {Eriksson, Edward},
  year = 2025,
  month = jan,
  journal = {arXiv:2501.09511},
  eprint = {2501.09511},
  primaryclass = {math},
  publisher = {arXiv},
  doi = {10.48550/arXiv.2501.09511},
  urldate = {2026-04-14},
  archiveprefix = {arXiv}
}

@article{Li2021,
  title = {Truncated Simulation and Inference in Edge-Exchangeable Networks},
  author = {Li, Xinglong and Campbell, Trevor},
  year = 2021,
  month = jan,
  journal = {Electronic Journal of Statistics},
  volume = {15},
  number = {2},
  pages = {5117--5157},
  publisher = {{Institute of Mathematical Statistics and Bernoulli Society}},
  issn = {1935-7524, 1935-7524},
  urldate = {2026-04-14},
  langid = {english}
}

@article{Campbell2018a,
  title = {Exchangeable Trait Allocations},
  author = {Campbell, Trevor and Cai, Diana and Broderick, Tamara},
  year = 2018,
  month = jan,
  journal = {Electronic Journal of Statistics},
  volume = {12},
  number = {2},
  pages = {2290--2322},
  publisher = {{Institute of Mathematical Statistics and Bernoulli Society}},
  issn = {1935-7524, 1935-7524},
  urldate = {2026-04-14},
  langid = {english}
}

@book{Mitzenmacher2005,
  title = {Probability and {{Computing}}: {{Randomized Algorithms}} and {{Probabilistic Analysis}}},
  shorttitle = {Probability and {{Computing}}},
  author = {Mitzenmacher, Michael and Upfal, Eli},
  year = 2005,
  month = jan,
  publisher = {Cambridge University Press},
  googlebooks = {0bAYl6d7hvkC},
  isbn = {978-0-521-83540-4},
  langid = {english}
}

@article{Freedman1975,
  title = {On {{Tail Probabilities}} for {{Martingales}}},
  author = {Freedman, David A.},
  year = 1975,
  month = feb,
  journal = {The Annals of Probability},
  volume = {3},
  number = {1},
  pages = {100--118},
  publisher = {Institute of Mathematical Statistics},
  issn = {0091-1798, 2168-894X},
  doi = {10.1214/aop/1176996452},
  urldate = {2026-05-31},
  langid = {english}
}

@article{Tropp2011,
  title = {Freedman's Inequality for Matrix Martingales},
  author = {Tropp, Joel},
  year = 2011,
  month = jan,
  journal = {Electronic Communications in Probability},
  volume = {16},
  number = {none},
  issn = {1083-589X},
  doi = {10.1214/ECP.v16-1624},
  urldate = {2026-05-31},
  langid = {english}
}

%%%%%%%%%%%%%%%%%%%%%%%%%%%%%%%%%%%%%%%%%%%%%%%%%%%%%%%%%%%%%%%%%%%%%%%%%%%%%%%
%%%%%%%%%%%%%%%%%%%%%%%%%%%%%%%%%%%%%%%%%%%%%%%%%%%%%%%%%%%%%%%%%%%%%%%%%%%%%%%
% APPENDIX
%%%%%%%%%%%%%%%%%%%%%%%%%%%%%%%%%%%%%%%%%%%%%%%%%%%%%%%%%%%%%%%%%%%%%%%%%%%%%%%
%%%%%%%%%%%%%%%%%%%%%%%%%%%%%%%%%%%%%%%%%%%%%%%%%%%%%%%%%%%%%%%%%%%%%%%%%%%%%%%
\newpage
\appendix
\onecolumn

% !TEX root = article.tex

\section{Asymptotic notation}
\label{sec:asymp}

Throughout this article, we write:
\begin{itemize}
    \item $f(x) \underset{x \to a}{=} O(g(x))$ if there exists $C > 0$ such that, for $x$ in a neighbourhood of $a$, $|f(x)| \le C|g(x)|$;
    \item $f(x) \underset{x \to a}{=} \Omega(g(x))$ if there exists $C > 0$ such that, for $x$ in a neighbourhood of $a$, $|f(x)| \ge C|g(x)|$;
    \item $f(x) \underset{x \to a}{=} \Theta(g(x))$ if both $f(x) \underset{x \to a}{=} O(g(x))$ and $f(x) \underset{x \to a}{=} \Omega(g(x))$ hold;
    \item $f(x) \underset{x \to a}{\sim} g(x)$ if $\lim_{x \to a} f(x)/g(x) = 1$.
\end{itemize}
where $a \in \mathbb{R} \cup \{-\infty, +\infty\}$. When $f(x)$ and $g(x)$ are random variables, the notation indicates that the relation holds almost surely. This usage of $\sim$ should not be confused with the standard distributional notation $X \sim \mathcal{D}$, which indicates that the random variable $X$ has distribution $\mathcal{D}$.

\section{Main lemma}

\begin{lemma}
\label{lem:xvsex}
Let $(\mathcal G_t)_{t\geq 0}$ be the edge-exchangeable multigraph sequence constructed from the weights $\mathcal W=\{w_i\}_{i\geq 1}$, and let $(N_t)_{t\geq 0}$ and $(N^{(e)}_t)_{t\geq 0}$ denote its associated vertex and edge counts. For $t\geq 1$, define the quenched means
\begin{align*}
\mu_{\mathcal W}(t)
    &:= \bbE\left(N_t\mid \mathcal W\right)
      = \sum_i \left[1-\prod_{j\neq i}(1-w_iw_j)^t\right],\\
\mu^{(e)}_{\mathcal W}(t)
    &:= \bbE\left(N_t^{(e)}\mid \mathcal W\right)
      = t\sum_{i<j} w_iw_j.
\end{align*}
Then, for almost every realization of $\mathcal W$ such that, for every
$\epsilon>0$,
\begin{align}
\sum_{t=1}^\infty
\exp\left\{
    -
    \frac{\epsilon^2\mu_{\mathcal W}(t)^2}
         {4(2\mu^{(e)}_{\mathcal W}(t)+\epsilon\mu_{\mathcal W}(t)/3)}
\right\}
    &<\infty, \label{eq:vertex-summability}
\end{align}
we have
\begin{align*}
N_t \overset{\mathrm{a.s.}}{\sim}
    \bbE\left(N_t\mid \mathcal W\right),
\qquad
N_t^{(e)} \overset{\mathrm{a.s.}}{\sim}
    \bbE\left(N_t^{(e)}\mid \mathcal W\right).
\end{align*}
\end{lemma}

\begin{proof}
Fix a realization of $\mathcal W$ such that
\[
\sum_i w_i<\infty.
\]
This event has probability one by the assumptions on $\nu$.
Throughout the proof, we work conditionally on this realization of
$\mathcal W$, so the numbers $(w_i)_{i\geq 1}$ are deterministic.

Write
\[
p_{ij}:=w_iw_j,\qquad i<j.
\]
Since $\sum_i w_i<\infty$, we have
\[
\sum_{i<j} p_{ij}
\leq \frac12\left(\sum_i w_i\right)^2<\infty.
\]
In particular, the conditional means below are finite for every fixed $t$.

Conditionally on $\mathcal W$, the edge multiplicities
$M^t_{\{i,j\}}$, $i<j$, are independent and satisfy
\[
M^t_{\{i,j\}}\sim \operatorname{Binom}(t,p_{ij}).
\]
Therefore
\[
N_t^{(e)}
=
\frac12\sum_{i\neq j} M^t_{\{i,j\}}
=
\sum_{i<j} M^t_{\{i,j\}}.
\]
Hence
\[
\mu^{(e)}_{\mathcal W}(t)
=
\bbE\left(N_t^{(e)}\mid \mathcal W\right)
=
t\sum_{i<j} p_{ij}.
\]

\paragraph{Edge count.}
For each $i<j$, write
\[
M^t_{\{i,j\}}
=
\sum_{s=1}^t B^{(s)}_{ij},
\]
where the variables $B^{(s)}_{ij}$ are independent Bernoulli random variables with success probability $p_{ij}$.
Thus, $N_t^{(e)}$ is a countable sum of independent Bernoulli random variables with total mean $\mu^{(e)}_{\mathcal W}(t)$. Applying \cref{thm:chernoff}, we obtain, for every $\epsilon\in(0,1)$,
\begin{align*}
\bbP\left(
    \left|N_t^{(e)}-\mu^{(e)}_{\mathcal W}(t)\right|
        >\epsilon \mu^{(e)}_{\mathcal W}(t)
    \,\middle|\,\mathcal W
\right)
&\leq
2\exp\left\{-\frac{\epsilon^2}{3}\mu^{(e)}_{\mathcal W}(t)\right\}\\
&\leq
2\exp\left\{-\frac{\epsilon^2}{3}t\sum_{i<j} p_{ij}\right\}.
\end{align*}
Since $\sum_{i<j}p_{ij}>0$ almost surely, the right-hand side is summable in $t$. Hence, by the Borel--Cantelli lemma,
\[
\frac{N_t^{(e)}}{\mu^{(e)}_{\mathcal W}(t)}
\longrightarrow 1
\]
almost surely, conditionally on $\mathcal W$.

\paragraph{Vertex count.}
The vertex indicators are not conditionally independent. Indeed, the indicators that vertices $i$ and $j$ are active both depend on the edge variable $M^t_{\{i,j\}}$. Therefore, one cannot prove the vertex part using the same Chernoff bound as for the edge count.

For $i\geq 1$, define
\[
A_{i,t}:=\mathbf 1\left\{\sum_{j\neq i} M^t_{\{i,j\}}>0\right\}.
\]
Then
\[
N_t=\sum_i A_{i,t},
\qquad
\mu_{\mathcal W}(t)
=
\sum_i \bbE(A_{i,t}\mid \mathcal W)
=
\sum_i \left[1-\prod_{j\neq i}(1-p_{ij})^t\right].
\]
We shall use a martingale bounded-differences argument.

First, truncate the edge set to the finite collection
\[
\mathcal I_m:=\{(i,j):1\leq i<j\leq m\}.
\]
Let $N_t(m)$ be the number of active vertices in the graph obtained by keeping only the edge indicators with $(i,j)\in\mathcal I_m$, and let
\[
\mu_{\mathcal W}(t,m)
    :=\bbE\left(N_t(m)\mid \mathcal W\right).
\]
The variables $\{M^t_{\{i,j\}}:(i,j)\in\mathcal I_m\}$ are independent.
Reveal them one at a time, in any deterministic order, and let $(Z_k)$ be the Doob martingale
\[
Z_k
:=
\bbE\left(
    N_t(m)
    \,\middle|\,
    \mathcal W,
    M^t_{e_1},\ldots,M^t_{e_k}
\right),
\]
where $e_1,\ldots,e_{|\mathcal I_m|}$ is the chosen ordering of $\mathcal I_m$.

Changing a single edge multiplicity can change the number of active vertices by at most $2$, because only the two endpoints of that edge can change their active/inactive status. Hence, the martingale increments satisfy
\[
|Z_k-Z_{k-1}|\leq 2.
\]

Moreover, if $e_k=(i,j)$, then the conditional variance of the $k$th martingale increment is bounded by
\[
\bbE[(Z_k-Z_{k-1})^2\mid Z_{0:k-1}]
    \leq 4tp_{ij}.
\]
Indeed, conditionally on the previously revealed variables, the only remaining randomness in this step is the binomial variable $M^t_{{i,j}}$. The functional $N_t(m)$ depends on this variable only through whether $M^t_{{i,j}}=0$ or $M^t_{{i,j}}>0$. Thus the conditional expectation after revealing $M^t_{{i,j}}$ can take two possible values, whose difference is at most $2$, since changing the status of the edge ${i,j}$ can affect only the two endpoints $i$ and $j$.

Therefore, the predictable quadratic variation of the martingale is bounded by
\[
\sum_{(i,j)\in\mathcal I_m} 4tp_{ij}
\leq
4\mu^{(e)}_{\mathcal W}(t).
\]

We can therefore apply Freedman's inequality (\cref{thm:freedman}) to the martingale
\[
Y_k:=Z_k-\mu_{\mathcal W}(t,m),
\]
with $R=2$ and $\sigma^2=4\mu^{(e)}_{\mathcal W}(t)$. For every $s>0$,
\[
\bbP\!\left(
N_t(m)-\mu_{\mathcal W}(t,m)\ge s
\mid \mathcal W
\right)
\leq
\exp\left\{
    -\frac{s^2}
    {2\left(4\mu^{(e)}_{\mathcal W}(t)+2s/3\right)}
\right\}.
\]

Applying the same argument to $-Y_k$ yields the lower-tail bound
\[
\bbP\!\left(
N_t(m)-\mu_{\mathcal W}(t,m)\le -s
\mid \mathcal W
\right)
\leq
\exp\left\{
    -\frac{s^2}
    {2\left(4\mu^{(e)}_{\mathcal W}(t)+2s/3\right)}
\right\}.
\]

Combining the two inequalities, we obtain
\[
\bbP\left(
    \left|N_t(m)-\mu_{\mathcal W}(t,m)\right|\geq s
    \,\middle|\, \mathcal W
\right)
\leq
2\exp\left\{
    -\frac{s^2}
    {2\left(4\mu^{(e)}_{\mathcal W}(t)+2s/3\right)}
\right\}.
\]

Since
\[
N_t(m)-\mu_{\mathcal W}(t,m)
\to
N_t-\mu_{\mathcal W}(t)
\qquad\text{almost surely},
\]
it follows that
\[
\{|N_t-\mu_{\mathcal W}(t)|> s\}
\subseteq
\liminf_{m\to\infty}
\{|N_t(m)-\mu_{\mathcal W}(t,m)|> s\}.
\]
Therefore, by Fatou's lemma,
\[
\bbP\!\left(
|N_t-\mu_{\mathcal W}(t)|> s
\,\middle|\,\mathcal W
\right)
\le
\liminf_{m\to\infty}
\bbP\!\left(
|N_t(m)-\mu_{\mathcal W}(t,m)|> s
\,\middle|\,\mathcal W
\right).
\]

Applying the previous Freedman bound yields
\[
\bbP\!\left(
|N_t-\mu_{\mathcal W}(t)|> s
\,\middle|\,\mathcal W
\right)
\le
2\exp\!\left\{
-\frac{s^2}
{2\bigl(4\mu^{(e)}_{\mathcal W}(t)+2s/3\bigr)}
\right\}.
\]

Taking $s=\epsilon \mu_{\mathcal W}(t)$, we obtain
\[
\bbP\left(
    \left|N_t-\mu_{\mathcal W}(t)\right|
        > \epsilon \mu_{\mathcal W}(t)
    \,\middle|\, \mathcal W
\right)
\leq
2\exp\left\{
    -
    \frac{\epsilon^2\mu_{\mathcal W}(t)^2}
         {2\left(4\mu^{(e)}_{\mathcal W}(t)+2\epsilon\mu_{\mathcal W}(t)/3\right)}
\right\}.
\]

By the summability assumption \eqref{eq:vertex-summability} and the Borel--Cantelli lemma,
\[
\frac{N_t}{\mu_{\mathcal W}(t)}
\longrightarrow 1
\]
almost surely, conditionally on $\mathcal W$.
\end{proof}

% !TEX root = article.tex

\section{Proof of \cref{thm:sparsity}}

Conditionally on the weights $\mathcal{W}$, our model is closely related to the rank-1 edge-exchangeable multigraph defined by \cite{Janson2018}. As a result, the proof strategy is formally parallel to that of \cite{Janson2018}: edge variables are independent, while vertex indicators are dependent only through shared incident edges.

We prove the following sharper result.

\begin{theorem}
\label{thm:sharp_sparsity}
Assume that, as \(w\downarrow0\), we have $\nu(w)\sim c\,w^{-2}\ell(w^{-1}),$
where \(c>0\), \(\ell\) is slowly varying, and
\[
\ell_1(x):=\int_x^\infty u^{-1}\ell(u)\,\mathrm du<\infty
\]
for all sufficiently large \(x\). Let
\[
S:=\sum_i w_i,
\qquad
L:=\sum_{i<j}w_iw_j.
\]
Then \(0<S<\infty\) and \(0<L<\infty\) a.s., and
\begin{align}
        N_t^{(e)}&\sim tL \qquad\text{a.s.}\label{eq:spar1}\\
        N_t&\sim cS\,t\ell_1(t) \qquad\text{a.s.} \label{eq:spar2}
\end{align}
Consequently, if \(\phi(x):=x\ell_1(x)\), then
\begin{equation}
N_t^{(e)}
=
\Theta\!\left(\phi^{-1}(N_t)\right)
\qquad\text{a.s.}\label{eq:spar3}
\end{equation}
If moreover $\ell_1(x\ell_1(x))\sim \ell_1(x),$
then
\begin{equation}
N_t^{(e)}
=
\Theta\!\left(\frac{N_t}{\ell_1(N_t)}\right)
\qquad\text{a.s.}\label{eq:spar4}
\end{equation}
In particular, if $\ell(x)=(\log x)^a$\ with $a<-1$, then
\begin{equation}
N_t^{(e)}
=
\Theta\!\left(N_t(\log N_t)^{-a-1}\right)
\qquad\text{a.s.}\label{eq:spar5}
\end{equation}
\end{theorem}

\begin{proof}
We write \(\mathcal W=\{w_i\}_{i\ge1}\) for the Poisson process of weights.
Since $\int_0^1 w\,\nu(\mathrm dw)<\infty,$ we have almost surely, $ 0<S:=\sum_iw_i<\infty$. Moreover,
\[
0<L:=\sum_{i<j}w_iw_j
=
\frac12\left(S^2-\sum_iw_i^2\right)
<\infty
\qquad\text{a.s.}
\]

\paragraph{Preliminary.} We first evaluate the asymptotic behaviour of the realised Poisson process.
\[
T_\lambda:=\sum_i \left(1-e^{-\lambda w_i}\right),
\qquad \lambda>0.
\]
Let us show that 
\begin{equation} 
        \label{eq:Tlambda_asymptotic} 
        T_\lambda\sim c\lambda\ell_1(\lambda) 
        \qquad\text{a.s.} 
\end{equation}
Let \(m(\lambda):=\mathbb E[T_\lambda]\). By Campbell's theorem,
\[
m(\lambda)
=
\int_0^\infty \left(1-e^{-\lambda w}\right)\nu(\mathrm dw).
\]
For any \(\varepsilon>0\), the integral over \([\varepsilon,\infty)\) is bounded by \(\nu([\varepsilon,\infty))<\infty\). Since \(\lambda\ell_1(\lambda)\to\infty\) as \(\lambda\to\infty\), this bounded term is \(o(\lambda\ell_1(\lambda))\). Thus, it is sufficient to analyse the behaviour near \(0\).

By hypothesis, for any \(\delta>0\), there exists \(\varepsilon>0\) such that for all \(w\in(0,\varepsilon]\):
\[
(1-\delta)cw^{-2}\ell(w^{-1}) \le \nu(w) \le (1+\delta)cw^{-2}\ell(w^{-1}).
\]
Let \(I(\lambda) = \int_0^\varepsilon \left(1-e^{-\lambda w}\right) w^{-2}\ell(w^{-1})\,\mathrm dw\). With the change of variables \(u=w^{-1}\), we obtain
\[
I(\lambda)
=
\int_{1/\varepsilon}^\infty
\left(1-e^{-\lambda/u}\right)\ell(u)\,\mathrm du.
\]
We partition this integral at \(u=\lambda\), assuming that \(\lambda\) is sufficiently large so that \(\lambda>1/\varepsilon\).

For the first part, \(u \in [1/\varepsilon, \lambda]\), we use the bound \(1-e^{-\lambda/u} \le 1\) to obtain
\[
\int_{1/\varepsilon}^{\lambda} \left(1-e^{-\lambda/u}\right)\ell(u)\,\mathrm du
\le
\int_{1/\varepsilon}^{\lambda}\ell(u)\,\mathrm du.
\]

Let us show that this upper bound is \(o(\lambda\ell_1(\lambda))\).First by \cref{thm:RVKaramata1}, 
$
\int_{1/\varepsilon}^{\lambda}\ell(u)\,\mathrm du \sim \lambda\ell(\lambda).
$
 We must now establish that \(\lim_{\lambda\to\infty} \ell(\lambda)/\ell_1(\lambda) = 0\). Recall that \(\ell_1(\lambda) = \int_\lambda^\infty u^{-1}\ell(u)\,\mathrm{d}u\), and by assumption, this integral is finite. According to \cref{prop:bingham159b}, since \(\ell\) is slowly varying and its tail integral \(\ell_1\) converges, \(\ell_1\) is also slowly varying and satisfies:
\[
\lim_{\lambda\to\infty} \frac{\ell_1(\lambda)}{\ell(\lambda)} = \lim_{\lambda\to\infty} \frac{1}{\ell(\lambda)} \int_\lambda^\infty \frac{\ell(u)}{u}\,\mathrm{d}u = \infty.
\]
By inverting this relationship, we conclude that the first integral is bounded by \(o(\lambda\ell_1(\lambda))\).

For the second part, \(u > \lambda\), we apply the inequalities \(x - x^2/2 \le 1-e^{-x} \le x\) for \(x \ge 0\). Setting \(x = \lambda/u\), we get:
\[
\int_{\lambda}^\infty \left(\frac{\lambda}{u} - \frac{\lambda^2}{2u^2}\right)\ell(u)\,\mathrm du \le \int_{\lambda}^\infty \left(1-e^{-\lambda/u}\right)\ell(u)\,\mathrm du \le \int_{\lambda}^\infty \frac{\lambda}{u}\ell(u)\,\mathrm du.
\]
The upper bound evaluates exactly to \(\lambda\ell_1(\lambda)\). For the lower bound, \cref{thm:RVKaramata1}, yields \(\int_{\lambda}^\infty u^{-2}\ell(u)\,\mathrm du \sim \lambda^{-1}\ell(\lambda)\). Therefore, the subtractive error term scales as \(\frac{\lambda^2}{2} O(\lambda^{-1}\ell(\lambda)) = O(\lambda\ell(\lambda))\), which is \(o(\lambda\ell_1(\lambda))\). 

Consequently, \[ I(\lambda)\sim\lambda\ell_1(\lambda), \] and it follows that \[ m(\lambda)\sim c\lambda\ell_1(\lambda). \]

To establish $T_\lambda\sim m(\lambda)$, we derive a specific concentration bound for \(T_\lambda\). Since \(T_\lambda \) is a functional of the Poisson point process \(\mathcal{W}\), its moment generating function is given by Campbell's theorem: for any \(u \in \mathbb{R}\),
\[
\mathbb{E}[\exp(u T_\lambda)] = \exp\left( \int_0^\infty \left(e^{u(1-e^{-\lambda w})} - 1\right) \nu(\mathrm{d}w) \right).
\]
For any \(w \ge 0\), the term \(x = 1 - e^{-\lambda w}\) lies in the interval \([0, 1]\). By the convexity of the function \(x \mapsto e^{ux}\), we have the bound \(e^{ux} - 1 \le x(e^u - 1)\). Applying this inequality to the integrand yields:
\[
\mathbb{E}[\exp(u T_\lambda)] \le \exp\left( (e^u - 1) \int_0^\infty (1 - e^{-\lambda w}) \nu(\mathrm{d}w) \right) = \exp\left( m(\lambda)(e^u - 1) \right).
\]
So the moment generating function of \(T_\lambda\) is bounded above by that of a Poisson random variable with parameter \(m(\lambda)\). Consequently, Chernoff's inequality yields, for any \(\eta\in(0,1)\), 
\[ 
\mathbb P\!\left( |T_\lambda-m(\lambda)| > \eta m(\lambda) \right) \le 2\exp\!\left( -\frac{\eta^2m(\lambda)}{3} \right). 
\]
Let us define a geometric grid \(\lambda_k=(1+\gamma)^k\) for a fixed \(\gamma>0\). Since \(m(\lambda)\) is regularly varying with index 1, it grows asymptotically faster than \(x^{1-\delta}\) for any \(\delta \in (0,1)\). Thus, \(m(\lambda_k)\) grows exponentially with \(k\), meaning the sequence of probabilities \(\mathbb{P}\left(|T_{\lambda_k} - m(\lambda_k)| > \eta m(\lambda_k)\right)\) decays exponentially and is therefore summable. By the Borel-Cantelli lemma, \(T_{\lambda_k} \sim m(\lambda_k)\) almost surely as \(k \to \infty\).

To extend this convergence to the continuous parameter \(\lambda\), we rely on the monotonicity of \(T_\lambda\). For any \(\lambda > 0\), there exists an index \(k\) such that \(\lambda_k \le \lambda \le \lambda_{k+1}\). As \(T_\lambda\) is nondecreasing and $m$ is strictly increasing in $\lambda$, we have 
\[
\frac{T_{\lambda_k}}{m(\lambda_{k+1})} \le \frac{T_\lambda}{m(\lambda)} \le \frac{T_{\lambda_{k+1}}}{m(\lambda_k)}.
\]
Taking the limit as \(\lambda \to \infty\) (and therefore \(k \to \infty\)), and utilizing the regular variation of \(m\) which guarantees \(\lim_{k\to\infty} m(\lambda_{k+1})/m(\lambda_k) = 1+\gamma\), we obtain almost surely:
\[
\frac{1}{1+\gamma} \le \liminf_{\lambda \to \infty} \frac{T_\lambda}{m(\lambda)} \le \limsup_{\lambda \to \infty} \frac{T_\lambda}{m(\lambda)} \le 1+\gamma.
\]
Since this holds for any arbitrarily small \(\gamma > 0\), taking the intersection of these almost sure events over a countable sequence \(\gamma \downarrow 0\) proves that \(T_\lambda \sim m(\lambda)\) a.s. This establishes \eqref{eq:Tlambda_asymptotic}.

\paragraph{Edge count.} By the first part of \cref{lem:xvsex}, we have established that, almost surely, 
\[ 
N_t^{(e)} \sim \bbE(N_t^{(e)}\mid\mathcal W) = tL. 
\] This is exactly \cref{eq:spar1}.

\paragraph{Vertex count.} We now study the number of active vertices.
Conditionally on \(\mathcal W\), the probability that vertex \(i\) is active after \(t\) interactions is
\[
p_i^{(t)}
=
1-\prod_{j\neq i}(1-w_iw_j)^t.
\]
Hence
$
\mu_{\mathcal W}(t)
:=
\mathbb E[N_t\mid \mathcal W]
=
\sum_i p_i^{(t)}.
$
We shall show that
\begin{equation}
\label{eq:conditional_mean_vertices}
\mu_{\mathcal W}(t)
\sim cS\,t\ell_1(t)
\qquad\text{a.s.}
\end{equation}

Fix \(\varepsilon>0\).
The number of atoms with \(w_i>\varepsilon\) is finite almost surely, and their total contribution to \(\mu_{\mathcal W}(t)\) is \(O(1)\), which is negligible compared to \(t\ell_1(t)\).
Thus, it is enough to consider atoms with \(w_i\le \varepsilon\). For such atoms, write
\[
R_i:=
-\sum_{j\neq i}\log(1-w_iw_j).
\]
Then $
p_i^{(t)}
=
1-e^{-tR_i}.
$ To bound \( R_i \), we use the inequalities \( x \le -\log(1-x) \le \frac{x}{1-x} \), valid for \( x \in [0, 1/2) \). Applying the lower bound with \( x = w_iw_j \) yields:
\[
R_i \ge \sum_{j \neq i} w_iw_j = w_i \sum_{j \neq i} w_j = w_i(S - w_i).
\]
For the upper bound, recall that we are only considering atoms where \( w_i \le \varepsilon \). Because the total sum of the weights is \( S \), every individual weight satisfies \( w_j \le S \). This implies that for any pair, the product is bounded by \( w_iw_j \le \varepsilon S \). Since \(S<\infty\) almost surely, we can choose \(\varepsilon\) sufficiently small such that \(\varepsilon S \le 1/2\). Applying the second inequality to the sum gives:
$$
R_i \le \sum_{j \neq i} \frac{w_iw_j}{1 - w_iw_j} \le \sum_{j \neq i} \frac{w_iw_j}{1 - \varepsilon S}\le \frac{1}{1 - \varepsilon S} \sum_{j \neq i} w_iw_j = \frac{1}{1 - \varepsilon S} w_i(S - w_i).
$$

In total 
\[
w_i(S-w_i)
\le
R_i
\le
\frac{1}{1-\varepsilon S}w_i(S-w_i),
\]
Since \(w_i\le \varepsilon\), it follows that \(S-w_i\ge S-\varepsilon\). Thus, 
\[
 R_i\ge (S-\varepsilon)w_i. 
 \]
We also have \(S-w_i\le S\), and since \(\varepsilon\downarrow0\), \[ \frac{1}{1-\varepsilon S} = 1+o_\varepsilon(1). \] Hence, 
\[
 R_i \le (1+o_\varepsilon(1))Sw_i. 
\]
As the function \( x \mapsto 1 - e^{-tx} \) is strictly monotonically increasing for \( t > 0 \), we have
\[
1 - e^{-t(S-\varepsilon)w_i} \le p_i^{(t)} \le 1 - e^{-t(1+o_\varepsilon(1))S w_i}.
\]
Summing these inequalities over all atoms satisfying \(w_i\le\varepsilon\), we obtain
\[
\sum_{w_i \le \varepsilon} \left(1 - e^{-t(S-\varepsilon)w_i}\right) \le \sum_{w_i \le \varepsilon} p_i^{(t)} \le \sum_{w_i \le \varepsilon} \left(1 - e^{-t(1+o_\varepsilon(1))S w_i}\right).
\]
Recall that \(\{w_i>\varepsilon\}\) is almost surely finite. Consequently, 
\begin{align*} 
        \sum_{w_i\le\varepsilon}p_i^{(t)} &= \mu_{\mathcal W}(t)+O(1),\\ 
        \sum_{w_i\le\varepsilon}(1-e^{-\lambda w_i}) &= T_\lambda+O(1), 
\end{align*} 
for every \(\lambda>0\) as $1-e^{-\lambda w_i}\leq 1$. By substituting \( \lambda = t(S-\varepsilon) \) into the left-hand side, \( \lambda = t(1+o_\varepsilon(1))S \) into the right-hand side, and incorporating the \( O(1) \) terms, we obtain the final bounded inequalities:
\[
T_{t(S-\varepsilon)} + O(1) \le \mu_{\mathcal W}(t) \le T_{t(1+o_\varepsilon(1))S} + O(1).
\]
Using \eqref{eq:Tlambda_asymptotic} and the slow variation of \(\ell_1\), we obtain 
\[ 
\liminf_{t\to\infty} \frac{\mu_{\mathcal W}(t)} {t\ell_1(t)} \ge c(S-\varepsilon), 
\] and 
\[ 
\limsup_{t\to\infty} \frac{\mu_{\mathcal W}(t)} {t\ell_1(t)} \le c(1+o_\varepsilon(1))S. 
\]
Letting \(\varepsilon\downarrow0\) proves
\eqref{eq:conditional_mean_vertices}.

Now, we apply \cref{lem:xvsex} conditionally on \(\mathcal W\). 
Recall that the summability condition requires \(\sum_{t=1}^\infty \exp\left\{ - \frac{\epsilon^2\mu(t)^2} {4(2\mu^{(e)}(t)+\epsilon/3\times\mu(t))} \right\} < \infty\). We know that \(\mu_{\mathcal W}(t) \sim cS\,t\ell_1(t)\) and \(\mu_{\mathcal W}^{(e)}(t) \sim tL\). Because \(\ell_1(t) \to 0\), \(\mu(t)\) grows strictly slower than \(\mu^{(e)}(t)\), meaning the denominator of the exponent is dominated by \(tL\). Consequently, the exponent scales asymptotically as \(-\Theta(t\ell_1(t)^2)\). 
To lower bound this growth, we apply \cref{lem:log} to the slowly varying function \(\ell_1\), which states that \(\lim_{n\to\infty} \frac{\ln(\ell_1(t))}{\ln(t)} = 0\). This means that for any \(\delta\in(0,1/2)\), and for all sufficiently large \(t\), \[ t\ell_1(t)^2 > t^{1-2\delta}. \] Thus the summability condition holds.
Therefore
\[
N_t\sim \mu_{\mathcal W}(t)
\sim cS\,t\ell_1(t)
\qquad\text{a.s.}
\]

\paragraph{Sparsity results.} We now prove \cref{eq:spar3}. Let
$
\phi(x):=x\ell_1(x).
$
Since \(\ell_1\) is slowly varying, \(\phi\) is regularly varying with index \(1\). According to \cref{thm:asymptotic_inverse}, any regularly varying function with index \(\alpha > 0\) possesses an asymptotic inverse that is regularly varying with index \(1/\alpha\). Therefore, \(\phi\) has an asymptotic inverse \(\phi^{-1}\) which is regularly varying with index \(1\), satisfying \(\phi^{-1}(\phi(x)) \sim x\) as \(x \to \infty\). We have proven 
\[
N_t\sim cS\,\phi(t)
\qquad\text{a.s.}
\]
Because \(\phi^{-1}\) is regularly varying with index \(1\), it preserves asymptotic equivalence and satisfies \(\phi^{-1}(ax) \sim a\phi^{-1}(x)\) for any constant \(a > 0\). Applying \(\phi^{-1}\) to both sides yields:
\[
\phi^{-1}(N_t) \sim \phi^{-1}(cS\phi(t)) \sim cS\phi^{-1}(\phi(t)) \sim cSt
\qquad\text{a.s.}
\]
Since \(c\) and \(S\) are strictly positive and finite almost surely, we can rearrange this asymptotic equivalence to isolate \(n\):
\[
t \sim \frac{1}{cS}\phi^{-1}(N_t)
\qquad\text{a.s.}
\]
This implies the exact asymptotic bound:
\[
t=\Theta\!\left(\phi^{-1}(N_t)\right)
\qquad\text{a.s.}
\]

Since \(N_t^{(e)}\sim tL\), it follows that
\[
N_t^{(e)}
=
\Theta\!\left(\phi^{-1}(N_t)\right)
\qquad\text{a.s.}
\]

We now prove \cref{eq:spar4}: assume that
$
\ell_1(x\ell_1(x))\sim \ell_1(x).
$
Since
$
N_t\sim cS\,t\ell_1(t),
$
slow variation gives
\[
\ell_1(N_t)\sim \ell_1(t\ell_1(t))\sim \ell_1(t).
\]
Therefore
\[
\frac{N_t}{\ell_1(N_t)}
\sim
\frac{cS\,t\ell_1(t)}{\ell_1(t)}
=
cS\,t.
\]
Since \(N_t^{(e)}\sim tL\), we conclude that
\[
N_t^{(e)}
=
\Theta\!\left(\frac{N_t}{\ell_1(N_t)}\right)
\qquad\text{a.s.}
\]

Finally, let us prove \cref{eq:spar5}. If \[ \ell(x) = (\log x)^a, \qquad a<-1, \] then
\[
\ell_1(x)
=
\int_x^\infty u^{-1}(\log u)^a\,\mathrm du
=
\frac{-1}{a+1}(\log x)^{a+1}.
\]
Moreover,
$
\ell_1(x\ell_1(x))\sim \ell_1(x).
$
Hence
\[
N_t^{(e)}
=
\Theta\!\left(\frac{N_t}{(\log N_t)^{a+1}}\right)
=
\Theta\!\left(N_t(\log N_t)^{-a-1}\right),
\]
\end{proof}
% !TEX root = article.tex

\section{Proof of \cref{thm:AsymptoticsLevyLaplace}}

When $\eta=1$, consider
\begin{align*}
(\alpha-\tau)x^{\alpha+1}\nu(x)  & =(1-x)^{\xi-1}\int_{\tau}^{\alpha}\frac{s}{\Gamma
(1-s)}x^{\alpha-s}\mathrm{d}s\\
& =(1-x)^{\xi-1}\int_{0}^{\alpha-\tau}\frac{\alpha-u}{\Gamma(1+u-\alpha)}e^{-u\ln(1/x)}\mathrm{d}u
\end{align*}
using the change of variables $u=\alpha-s$. Let $g(z)=\int_{0}^{\alpha-\tau}\frac{\alpha-u}{\Gamma(1+u-\alpha)}e^{-u z}\mathrm{d}u$. We have, as $u\to 0$,
\begin{align*}
\frac{\alpha-u}{\Gamma(1+u-\alpha)}\sim\left \{\begin{array}{ll}
                                     u & \text{if }\alpha=1, \\
                                     \frac{\alpha}{\Gamma(1-\alpha)} & \text{if }\alpha<1 .
                                   \end{array}\right .
\end{align*}
Using the Tauberian \cref{thm:RVTauberian}, we obtain, as $z\to\infty$
$$
g(z)\sim\left \{\begin{array}{ll}
                                     z^{-2} & \text{if }\alpha=1, \\
                                     \frac{z^{-1}\alpha}{\Gamma(1-\alpha)} & \text{if }\alpha<1.
                                   \end{array}\right .
$$
It follows that, as $x\to 0$,
$$
(\alpha-\tau)x^{\alpha+1}\nu(x)=g(\ln 1/x)\sim\left \{\begin{array}{ll}
                                     \ln^{-2}(1/x) & \text{if }\alpha=1, \\
                                     \frac{\ln^{-1}(1/x) \alpha}{\Gamma(1-\alpha)} & \text{if }\alpha<1.
                                   \end{array}\right .
$$

The result when $\eta>0$ follows immediately as $\eta$ is a multiplicative constant.

\section{Conditions}

We now prove that the RapidBeta measure satisfies the assumption $\nu([0,1])=\infty$ and $\int_0^1w\nu(\mathrm{d}w)<\infty$.

\paragraph{Divergence of $\int_0^1 \nu(w)\mathrm{d}w$}

The integral is:
\[
\int_0^1 \nu(w)\mathrm{d}w = \int_0^1 \left( \frac{\eta}{\alpha-\tau} \int_{\tau}^{\alpha} \frac{s}{\Gamma(1-s)} w^{-1-s} (1-w)^{\xi-1} ds \right) \mathrm{d}w
\]
By Fubini's theorem, we can swap the order of integration:
\[
\frac{\eta}{\alpha-\tau} \int_{\tau}^{\alpha} \frac{s}{\Gamma(1-s)} \left( \int_0^1 w^{-1-s} (1-w)^{\xi-1} \mathrm{d}w \right) ds
\]
The inner integral over $w$ is the Beta function, $B(x,y) = \int_0^1 t^{x-1}(1-t)^{y-1}\mathrm{d}t$. We identify the parameters $x$ by setting $x-1 = -1-s$, which gives $x=-s$.

The Beta integral converges only if its parameters are positive. This requires $x = -s > 0$, which implies $s < 0$.

However, the given constraint is $0 \le \tau < \alpha \le 1$. For any $s \in [\tau, \alpha]$, we have $s \ge 0$. This violates the necessary condition $s < 0$. Because the inner integral
\[
\int_0^1 w^{-1-s} (1-w)^{\xi-1} \mathrm{d}w
\]
diverges for all $s \ge 0$ in the domain, the entire expression for $\int_0^1 \nu(w)\mathrm{d}w$ diverges.

\paragraph{Convergence of $\int_0^1 w\nu(w)\,\mathrm{d}w$.} We want to show that
\(
    \int_0^1 w\nu(w)\,\mathrm{d}w<\infty .
\)
By definition of \(\nu\), we have
\[
\int_0^1 w\nu(w)\,\mathrm{d}w
=
\int_0^1
w\left(
\frac{\eta}{\alpha-\tau}
\int_{\tau}^{\alpha}
\frac{s}{\Gamma(1-s)}
w^{-1-s}(1-w)^{\xi-1}
\,\mathrm{d}s
\right)
\,\mathrm{d}w .
\]
Combining the powers of \(w\) gives
\[
\int_0^1 w\nu(w)\,\mathrm{d}w
=
\frac{\eta}{\alpha-\tau}
\int_0^1
\int_{\tau}^{\alpha}
\frac{s}{\Gamma(1-s)}
w^{-s}(1-w)^{\xi-1}
\,\mathrm{d}s\,\mathrm{d}w .
\]
By Tonelli's theorem,
\[
\int_0^1 w\nu(w)\,\mathrm{d}w
=
\frac{\eta}{\alpha-\tau}
\int_{\tau}^{\alpha}
\frac{s}{\Gamma(1-s)}
\left(
\int_0^1
w^{-s}(1-w)^{\xi-1}
\,\mathrm{d}w
\right)
\,\mathrm{d}s .
\]
For \(s<1\), the inner integral is the Beta integral with parameters
\(1-s\) and \(\xi\). Since \(1-s>0\) and \(\xi>0\), it is finite and satisfies
\[
\int_0^1
w^{-s}(1-w)^{\xi-1}
\,\mathrm{d}w
=
B(1-s,\xi)
=
\frac{\Gamma(1-s)\Gamma(\xi)}
{\Gamma(1-s+\xi)} .
\]
Therefore, for \(s<1\),
\[
\frac{s}{\Gamma(1-s)}
\int_0^1
w^{-s}(1-w)^{\xi-1}
\,\mathrm{d}w
=
\frac{s}{\Gamma(1-s)}
\frac{\Gamma(1-s)\Gamma(\xi)}
{\Gamma(1-s+\xi)}
=
\frac{s\Gamma(\xi)}
{\Gamma(\xi+1-s)} .
\]
Thus,
\[
\int_0^1 w\nu(w)\,\mathrm{d}w
=
\frac{\eta\Gamma(\xi)}{\alpha-\tau}
\int_{\tau}^{\alpha}
\frac{s}
{\Gamma(\xi+1-s)}
\,\mathrm{d}s .
\]
This identity is immediate when \(\alpha<1\). If \(\alpha=1\), the Beta integral itself is singular at the single point \(s=1\), but this point has Lebesgue measure zero. Moreover, 
\(
s\mapsto \frac{s\Gamma(\xi)}{\Gamma(\xi+1-s)}
\)
has a continuous extension to \(s=1\), since
\[
\lim_{s\uparrow 1}
\frac{s\Gamma(\xi)}
{\Gamma(\xi+1-s)}
=
\frac{\Gamma(\xi)}{\Gamma(\xi)}
=
1 .
\]
Hence the endpoint \(s=1\) does not affect the value or finiteness of the integral. Finally, the function
\(
s\mapsto \frac{s}{\Gamma(\xi+1-s)}
\)
is continuous on the closed interval \([\tau,\alpha]\). Indeed, for all
\(s\in[\tau,\alpha]\),
\[
\xi+1-s\geq \xi>0,
\]
and the gamma function is finite, strictly positive, and continuous on
\((0,\infty)\). Consequently,
\[
\int_{\tau}^{\alpha}
\frac{s}{\Gamma(\xi+1-s)}
\,\mathrm{d}s<\infty .
\]
Therefore,
\[
\int_0^1 w\nu(w)\,\mathrm{d}w
=
\frac{\eta\Gamma(\xi)}{\alpha-\tau}
\int_{\tau}^{\alpha}
\frac{s}
{\Gamma(\xi+1-s)}
\,\mathrm{d}s
<\infty .
\]

% !TEX root = article.tex

\section{Sampling algorithm}
\label{sec:weights}

\subsection{Proofs}

\subsubsection{Proof of \cref{thm:correctness}}

On each region, $A_\varepsilon$ and $B$, the proposal intensity dominates the target intensity. By the thinning theorem for Poisson point processes \citep[Section 5.1]{Kingman1993}, the accepted points on each region form independent Poisson point processes with an intensity equal to the target intensity $\lambda(\mathrm{d}s, \mathrm{d}w)$ restricted to that region.

Since $A_\varepsilon$ and $B$ are disjoint, their superposition yields a Poisson point process on $[\tau, \alpha] \times [\varepsilon, 1]$ with intensity $\lambda(\mathrm{d}s, \mathrm{d}w)$. Finally, marginalising out the latent dimension $s$ (which corresponds to the projection property of Poisson processes) yields a Poisson point process on $[\varepsilon, 1]$ with intensity:
\[
\nu(w)\,\mathrm{d}w = \left( \int_{\tau}^{\alpha} \lambda(s, w) \,\mathrm{d}s \right) \mathrm{d}w.
\]
This establishes that $\mathcal{W}_\varepsilon$ is a realisation of the $\varepsilon$-truncated RapidBeta process.

\subsubsection{Proof of \cref{prop:discarded_mass}}

By Campbell's theorem for Poisson point processes,
\[
\mathbb{E}[R_\varepsilon]
=
\int_0^\varepsilon w\,\nu(w)\,\mathrm{d}w.
\]
Using \cref{thm:AsymptoticsLevyLaplace} with $\alpha=1$, there exist
constants $0<c_1<c_2<\infty$ and $\varepsilon_0\in(0,1)$ such that for all
$0<w<\varepsilon_0$,
\[
c_1\,w^{-2}\big(\ln(1/w)\big)^{-2}
\le
\nu(w)
\le
c_2\,w^{-2}\big(\ln(1/w)\big)^{-2}.
\]
Hence, for all sufficiently small $\varepsilon$,
\[
c_1 \int_0^\varepsilon \frac{1}{w\big(\ln(1/w)\big)^2}\,\mathrm{d}w
\le
\mathbb{E}[R_\varepsilon]
\le
c_2 \int_0^\varepsilon \frac{1}{w\big(\ln(1/w)\big)^2}\,\mathrm{d}w.
\]
Now make the change of variables $u=\ln(1/w)$, so that
$\mathrm{d}u = -\mathrm{d}w/w$. Then
\[
\int_0^\varepsilon \frac{1}{w\big(\ln(1/w)\big)^2}\,\mathrm{d}w
=
\int_{\ln(1/\varepsilon)}^\infty u^{-2}\,\mathrm{d}u
=
\frac{1}{\ln(1/\varepsilon)}.
\]
Combining the upper and lower bounds yields
\[
\mathbb{E}[R_\varepsilon]
=
\Theta\!\left(\frac{1}{\ln(1/\varepsilon)}\right),
\]
as claimed.
% !TEX root = article.tex

\section{Background on rapidly varying function}
\label{sec:rapidvary}

Most of the background material in this section originates from the book of \citet{Bingham1987}.

\subsection{Definitions}
\begin{definition}[Slowly varying function]
A function $\ell:(0,\infty)\to(0,\infty)$ is slowly varying at infinity if for all $c>0$,
$$
\frac{\ell(cx)}{\ell(x)}\to 1\text{  as }x\to\infty.
$$
\end{definition}
\noindent Examples of slowly varying functions include $\ln^a$, for $a\in\Real$, and functions converging to a constant $c>0$.

\begin{definition}[Regularly varying function]
A function $\ell:(0,\infty)\to(0,\infty)$ is regularly varying at infinity with exponent $\rho\in\Real$ if $f(x)=x^\rho\ell(x)$ for some slowly varying function $\ell$. We note $R_\rho$ the set of all functions regularly varying at infinity with exponent $\rho$. A function $f$ is regularly varying at
$0$ if $f(1/x)$ is regularly varying at infinity that is, $f(x)=x^{-\rho}\ell(1/x)$ for some $\rho\in\Real$. 
\end{definition}

\subsection{Karamata theorems}

The following propositions and corollaries relate to integrals of regularly varying functions.

\begin{proposition}[Karamata theorem] See \citep[Propositions 1.5.8 and 1.5.10]{Bingham1987}. Let $U(t)=t^\rho\ell(t)$ for some locally bounded slowly varying function $\ell$. Then
\begin{itemize}
\item[(i)] If $\rho>-1$
$$
\int_0^x U(t)dt \sim\frac{1}{1+\rho}x^{\rho +1}\ell(x)\text{  as }x\to\infty.
$$
\item[(ii)] If $\rho<-1$
$$
\int_x^\infty U(t)dt\sim -\frac{1}{1+\rho}x^{\rho +1}\ell(x)\text{  as }x\to\infty.
$$
\end{itemize}
\label{thm:RVKaramata1}
\end{proposition}
The following corollaries will be useful.

\begin{proposition}
\label{thm:RVKaramata2}
Let $U(x)=x^\alpha \ell(1/x)$ for some locally bounded slowly varying function $\ell$.
\begin{itemize}
\item[(i)] If $\alpha>-1$
$$
\int_0^x U(t)\mathrm{d}t\sim \frac{1}{\alpha+1}x^{1+\alpha}\ell(1/x)\text{  as }x\to 0.
$$
\item[(ii)] If $\alpha<-1$
$$
\int_x^\infty U(t)\mathrm{d}t\sim -\frac{1}{\alpha+1}x^{1+\alpha}\ell(1/x)\text{  as }x\to 0.
$$

\end{itemize}
If $\ell$ is slowly varying and $\alpha>-1$ then
$\int_{0}^{x}t^{\alpha}\ell(1/t)dt$ converges and%
\[
\frac{x^{1+\alpha}\ell(1/x)}{\int_{0}^{x}t^{\alpha}\ell(1/t)dt}\rightarrow
\alpha+1\text{ as }x\rightarrow0.
\]

\end{proposition}

\begin{proof}
Let $\rho<-1$. We have $\int
_{x}^{\infty}t^{\rho}\ell(t)\mathrm{d}t=\int_{0}^{1/x}t^{-\rho-2}\ell(1/t)\mathrm{d}t$. Writing
$\alpha=-\rho-2>-1$, we obtain, using \cref{thm:RVKaramata1}(ii) %
\[
\frac{x^{-\alpha-1}\ell(x)}{\int_{0}^{1/x}t^{\alpha}\ell(1/t)\mathrm{d}t}%
\rightarrow\rho+1\text{ as }x\rightarrow\infty,
\]
or
\[
\frac{x^{1+\alpha}\ell(1/x)}{\int_{0}^{x}t^{\alpha}\ell(1/t)dt}\rightarrow
\rho+1\text{ as }x\rightarrow0.
\]
The case $\alpha<-1$ follows similarly, using \cref{thm:RVKaramata1}(i).
\end{proof}

\begin{proposition}[Proposition 1.5.9b, \citealp{Bingham1987}]
\label{prop:bingham159b}
Let $\ell$ be slowly varying and suppose $\int_x^\infty \ell(t)\mathrm{d}t/t < \infty$. Then $\int_x^\infty \ell(t)\mathrm{d}t/t$ is slowly varying and
\[
\frac{1}{\ell(x)} \int_x^\infty \ell(t)\mathrm{d}t/t \to \infty \quad \text{as } x \to \infty.
\]
\end{proposition}

\subsection{Asymptotic inverse}

\begin{theorem}[Theorem 1.5.12, \citealp{Bingham1987}]
\label{thm:asymptotic_inverse}
Let $f \in R_\alpha$ with $\alpha > 0$. Then there exists $g \in R_{1/\alpha}$ with
\[
f(g(x)) \sim g(f(x)) \sim x \quad \text{as } x \to \infty.
\]
Here $g$ is unique to within asymptotic equivalence, and one version of $g$ is the generalized inverse $f^{-1}(x) = \inf\{y: f(y) > x\}$.
\end{theorem}

\subsection{Tauberian theorem}

The following theorem is a variation of \citet[Theorem 1.7.6 p.46]{Bingham1987}, where the two limits at 0 and infinity are exchanged. The proof is similar. See also \citet[Chapter XIII]{Feller1971}.

\begin{theorem}[Tauberian theorem]\label{thm:RVTauberian}
Assume $U(x)\geq0$, $c\geq0$, $\rho>-1$, $\widehat{U}(s)=s\int_{0}^{\infty
}e^{-sx}U(x)\mathrm{d}x$ convergent for $s>0$, and $\ell$ a slowly varying function.
Then%
\[
U(x)\sim cx^{\rho}\ell(1/x)/\Gamma(1+\rho)\text{ as }x\rightarrow0
\]
implies%
\[
\widehat{U}(s)\sim cs^{-\rho}\ell(s)\text{ as }s\rightarrow\infty.
\]

\end{theorem}

\begin{proof}
Write $V(x)=\int_{0}^{x}U(y)\mathrm{d}y$ (this is finite for any $x$ as $\widehat
{U}(s)$ is convergent for any $s$), then $V$ is non-decreasing and by
\cref{thm:RVKaramata2}
\[
V(x)\sim\frac{c}{\rho+1}x^{\rho+1}\ell(1/x)/\Gamma(1+\rho)\text{ as
}x\rightarrow0.
\]
Then by \citet[Theorem 1.7.1, p.38]{Bingham1987}, this is equivalent to
\[
\widehat{V}(s)=\int_{0}^{\infty}e^{-sx}\mathrm{d}V(x)\sim cs^{-\rho-1}\ell(s)\text{ as
}s\rightarrow\infty.
\]
Finally, note that $\widehat{V}(s)=\frac{\widehat{U}(s)}{s}$. Thus the above equation is
equivalent to%
\[
\widehat{U}(s)\sim cs^{-\rho}\ell(s)\text{ as }s\rightarrow\infty.
\]
\end{proof}

\subsection{Other useful results on regular variation}

\begin{lemma}[Lemma 14, \citealp{Gnedin2007}]
    \label{gnedin14}
For $\ell$ slowly varying, the relation 

$$\int_x^\infty \nu(du) \sim x^{-1}\ell(x^{-1}), ~~x\rightarrow0$$

implies 

$$\int_0^x u\nu(du) \sim \ell_1(x^{-1}), ~~x\rightarrow0$$

with $\ell_1(x)=o(\ell(x))$ is another slowly varying function defined by $\ell_1(y)=\int_y^\infty u^{-1}\ell(u)du$.
\end{lemma}

\begin{lemma}[Proposition 1.3.6, \citealp{Bingham1987}]
\label{lem:log}
If $\ell$ varies slowly then $\frac{\ln(\ell(x))}{\ln(x)} \underset{x\rightarrow\infty}{\longrightarrow} 0$.
\end{lemma}

\section{More background results}

\begin{theorem}
\label{thm:chernoff}
Let $X_1,X_2,\ldots$ be a sequence of independent Bernoulli random variables with parameters $p_1,p_2,\ldots$ such that $\mu=\sum_k p_k <\infty$. Let
\[
S=\sum_k X_k.
\]
Then, for every $\epsilon\in(0,1)$,
\[
\bbP(|S-\mu|>\epsilon\mu)\leq 2\exp\left(-\mu\frac{\epsilon^2}{3}\right).
\]
\end{theorem}

\begin{proof}
For $n>0$, let
\[
S_n=\sum_{k=1}^n X_k
\qquad\text{and}\qquad
\mu_n=\bbE(S_n)=\sum_{k=1}^n p_k.
\]
The sequence $(\mu_n)$ converges to $\mu$.

Since
\[
\sum_{k=1}^\infty \bbP(X_k=1)=\mu <\infty,
\]
the first Borel--Cantelli lemma implies that the events $\{X_k=1\}$ occur only finitely often almost surely. Consequently, $S$ is almost surely finite, and the partial sums $S_n$ converge almost surely to $S$.

Fix $\epsilon\in(0,1)$ and consider the events
\begin{align*}
    A_n&=\{|S_n-\mu_n|\geq \epsilon\mu\},\\
    A&=\{|S-\mu|\geq \epsilon\mu\}.
\end{align*}

For any sample path such that $S_n(\omega)\to S(\omega)$, if $|S-\mu|\geq \epsilon\mu$, then for all sufficiently large $n$,
\[
|S_n-\mu_n|\geq \epsilon\mu
\]
also holds, since $\mu_n\to\mu$. Therefore, the indicator functions satisfy
\[
\mathbf{1}_{A}\leq \liminf_{n\to\infty} \mathbf{1}_{A_n}
\qquad \text{almost surely}.
\]

Taking expectations on both sides and applying Fatou's lemma yields
\[
\bbP(A)
=\bbE[\mathbf{1}_{A}]
\leq \bbE\!\left[\liminf_{n\to\infty} \mathbf{1}_{A_n}\right]
\leq \liminf_{n\to\infty} \bbE[\mathbf{1}_{A_n}]
=\liminf_{n\to\infty}\bbP(A_n).
\]

By Corollary 4.6 of \citet{Mitzenmacher2005}, for every $n>0$,
\[
\bbP(A_n)\leq 2\exp(-\mu_n\epsilon^2/3).
\]
Consequently,
\[
\liminf_{n\to\infty}\bbP(A_n)
\leq
\liminf_{n\to\infty}2\exp(-\mu_n\epsilon^2/3)
=
\lim_{n\to\infty}2\exp(-\mu_n\epsilon^2/3)
=
2\exp(-\mu\epsilon^2/3).
\]
Combining the previous inequalities completes the proof.
\end{proof}

\begin{theorem}[\citealt{Tropp2011}, Theorem 1.1; \citealt{Freedman1975}, Theorem 1.6]
\label{thm:freedman}
Consider a real-valued martingale $(Y_k)_{k\geq0}$ with difference sequence $(X_k)_{k\geq0}$, and assume that the difference sequence is uniformly bounded:
\[
X_k \leq R
\qquad \text{almost surely for all } k\geq 0.
\]

Define the predictable quadratic variation process by
\[
W_k:=\sum_{j=1}^k \mathbb{E}\!\left[X_j^2\,\middle|\,X_{0:j-1}\right],
\qquad k\geq 0.
\]

Then, for all $t\geq 0$ and $\sigma^2>0$,
\[
\mathbb{P}\left\{
\exists k\geq 0:
Y_k\geq t
\text{ and }
W_k\leq \sigma^2
\right\}
\leq
\exp\left(
-\frac{t^2/2}{\sigma^2+Rt/3}
\right).
\]
\end{theorem}

\end{document}